\documentclass{amsart}
\usepackage{pinlabel}

\usepackage{amssymb}
\usepackage{amssymb}

\usepackage{amsmath,amscd,amsthm}
\usepackage[mathscr]{euscript}
\usepackage[all]{xy}
\usepackage[utf8]{inputenc}
\usepackage{lmodern}

\usepackage{mathrsfs}
\usepackage{tikz}
\usetikzlibrary{cd}
\usetikzlibrary{arrows,automata, positioning}
\usetikzlibrary{matrix,arrows,decorations.pathmorphing}

\newcommand{\adju}[4]{\xymatrix{#1:#2\ar@<.5ex>[r]^{}&#3:#4\ar@<.5ex>[l]^{}}}
\newcommand{\adj}[2]{\ensuremath{({#1}\dashv{#2})}}

\newcommand{\tT}{\mathbf{Top}}
\newcommand{\Ab}{\mathbf{Ab}}

\newcommand{\sS}{\mathbf{sSet}}
\newcommand{\Set}{\mathbf{Set}}
\newcommand{\Cat}{\mathbf{Cat}}

\newcommand{\Mon}{\mathbf{Mon}}
\newcommand{\Grp}{\mathbf{Grp}}

\DeclareMathOperator{\Res}{Res}
\DeclareMathOperator{\Ind}{Ind}

\DeclareMathOperator*{\colim}{colim}
\DeclareMathOperator*{\hocolim}{hocolim}

\DeclareMathOperator{\Hom}{Hom}

\DeclareMathOperator{\opp}{op}

\DeclareMathOperator{\Simp}{\mathbf{\Delta}}

\DeclareMathOperator{\Fix}{Fix}
\DeclareMathOperator{\Tot}{Tot}
\DeclareMathOperator{\diag}{\mathbf{diag}}

\newcommand{\rnv}[1]{\ensuremath{{\textbf{\textit{q}}_*^{#1}}}}
\newcommand{\rinv}[2]{\ensuremath{{\textbf{\textit{q}}_*^{#1}{(#2)}}}}

\newcommand{\hrnv}[2]{\ensuremath{{\Hom_M({G(#1)},{#2})}}}

\newcommand{\bB}{\mathbf B}

\newcommand{\into}{\hookrightarrow}
\newcommand{\bbZ}{\mathbb Z}
\newcommand{\bbN}{\mathbb N} 
\newcommand{\cC}{\mathcal C}
\newcommand{\frC}{\bf{\mathfrak C}}
\newcommand{\frD}{\bf{\mathfrak D}}
\newcommand{\frX}{\mathfrak X}

\newcommand{\ii}{\mathcal I}
\newcommand{\jj}{\mathcal J}

\newcommand{\cY}{\mathcal Y}
\newcommand{\cX}{\mathcal X}
\newcommand{\bbR}{\mathbb R}
\newcommand{\bbS}{\mathbb S}
\newcommand{\bbD}{\mathbb D}
\newcommand{\cV}{\mathcal V}
\newcommand{\cO}{{\mathbf O}}

 \newtheorem{theorem}{Theorem}[section]
 \newtheorem{lemma}[theorem]{Lemma}
 \newtheorem{proposition}[theorem]{Proposition}
 \newtheorem{corollary}[theorem]{Corollary}
  \newtheorem{example}[theorem]{Example}
  
 \theoremstyle{definition}
 
\newtheorem{remark}[theorem]{Remark}
\newtheorem{notation}[theorem]{Notation}%

\begin{document}

\title[An Elmendorf-Piacenza style Theorem for Monoid actions]{Homotopy theory of monoid actions via group actions and an Elmendorf style theorem}

\author{Mehmet Akif Erdal}

\address{Department of Mathematics, Faculty of Arts and Sciences,
	Yeditepe University, 34736, İstanbul, Turkiye}

\email{mehmet.erdal@yeditepe.edu.tr}

\begin{abstract}	Let $M$ be a monoid and $G:\mathbf{Mon} \to \mathbf{Grp}$ be the group completion functor from monoids to groups. Given a collection $\mathcal{X}$ of submonoids of $M$ and for each $N\in \mathcal{X}$ a collection $\mathcal{Y}_N$ of subgroups of $G(N)$, we construct a model structure on the category  of $M$-spaces and $M$-equivariant maps, called the $(\cX,\cY)$-model structure, in which weak equivalences and fibrations are induced from the standard $\mathcal{Y}_N$-model structures on $G(N)$-spaces for all $N\in \mathcal{X}$. We also show that for a pair of collections $(\mathcal{X},\mathcal{Y})$ there is a  small category ${\mathbf O}_{(\mathcal{X},\mathcal{Y})}$  whose objects are $M$-spaces $M\times_NG(N)/H$ for each $N\in \cX$ and $H\in \cY_N$  and morphisms are $M$-equivariant maps, such that the $(\cX,\cY)$-model structure on the category of $M$-spaces is Quillen equivalent to the projective model structure on the category of contravariant ${\mathbf O}_{(\mathcal{X},\mathcal{Y})}$-diagrams of spaces.\end{abstract}


\maketitle

\section{Introduction}\label{sec:intro}
For a group $G$ and a collection $\cY$ of subgroups of $G$, the orbit category $\cO_{\cY}$ has  $G$-orbits $G/H$ for $H\in \cY$ as objects and $G$-equivariant maps between them as morphisms. Due to Elmendorf's Theorem  \cite{elmendorf} the category $G$-spaces and the category of contravariant  $\cO_{\cY}$-diagrams of spaces admit equivalent homotopy theories. This provides a great convenience when studying $G$-equivariant homotopy theory since one can reduce it to non-equivariant homotopy theory of associated fixed point systems. The main objective of this paper is to prove a similar result for monoid actions, which reduces the homotopy theory of monoid actions to standard equivariant  homotopy theory of universally associated group actions.

Let $M$ be a discrete monoid and let $q_M:M\to G(M)$ be the universal morphism of its group completion.  Denote by $M\tT$ and $G(M)\tT$ categories of $M$-spaces and $G(M)$-spaces, respectively. The morphism $q_M$ induces a functor $q^*_M:G(M)\tT\to M\tT$ which sends a $G(M)$-space to the same space with the restricted $M$-action.  Due to the universal property of $q_M$, any $M$-space on which $M$ acts by homeomorphisms is isomorphic to an object in the image of $q^*_M$. The functor $q^*_M$ admits a right adjoint, which we denote by ${\rnv{M}}$. For an $M$-space $X$, $\rinv{M}{X}$ can be given as the hom-space $\hom_M(G(M),X)$ of $M$-equivariant maps from $G(M)$ to $X$, where $G(M)$ is considered with the left $M$-action induced by translation after applying $q_M$. The $G(M)$-action on $\hom_M(G(M),X)$ is the one induced by the right-translation on $G(M)$. This functor sorts out ``symmetries'' of $M$-spaces; that is, $M$-invariant subspaces on which $M$ acts by homeomorphisms and maximal with respect to this property.

For each submonoid $N\leq M$, let $q^*_N$ denote the functor induced by $q_N$ and $\rnv{N}$ its right adjoint. We first introduce a category $\cO(M:q)$, which we call \emph{the induced orbit category of $M$}, as the full-subcategory of $M\tT$ of all induced orbits $M\times_N q^*_N(G(N)/H)$ for all $N\leq M$ and $H\leq G(N)$. For simplicity we denote the object $M\times_N q^*_N(G(N)/H)$ in $\cO(M:q)$ by  $(N,H)$.  

By a \emph{pair of collections} $(\cX,\cY)$ we mean the following data: $\cX$ is a family of submonoids of $M$ containing the trivial submonoid, $\cY$ is a set of collections $\cY_N$ of subgroups of $G(N)$, each containing the trivial subgroup for every $N\in \cX$; such that, the full subcategory $\cO_{(\cX,\cY)}$  of all objects $(N,H)$ with $N \in \cX$ and $H \in \cY_N$ is replete in $\cO(M:q)$. Given a pair of collections, there exists  a pair of adjoint functors $$\adju{\Upsilon}{[{\cO_{(\cX,\cY)}}^{\opp},\tT]}{M\tT}{\frX}$$ where $\frX (X)(N,H)=\rinv{N}{\Res^M_N(X})^H$ and $\Upsilon(F)=F((\mathbf{e},\mathbf{e}))$, see Section \ref{sect:model-structures-on-m-spaces}. Our main theorem is the following.
\begin{theorem}\label{thm:Rmodel} Let $\cX$ be a collection of submonoids of $M$ and for each $N\in \cX$, $\cY_N$ be a collection of subgroups of $G(N)$. Then $M\tT$ admits a cofibrantly generated model structure in which a map $f$ is a weak equivalence (resp. fibration) if for each $N\in \cX$ and $H\in \cY_N$, $(\rinv{N}{\Res^M_N(f)})^H$ is a weak homotopy equivalence (resp. Serre fibration). The generating cofibrations and acyclic cofibrations are respectively given by $$M\ii_{(\cX,\cY)}=\{M\times_N q^*_N(G(N)/H)\times i_n \ {\mid} \  N\in \cX, H\in \cY_N, n\in \bbN \}$$ and
	$$M\jj_{(\cX,\cY)}=\{M\times_N q^*_N(G(N)/H)\times j_n \ {\mid} \  N\in \cX, H\in \cY_N, n\in \bbN\},$$
	where  $i_n:\bbS^{n-1}\to \bbD^n$ and $j_n:\bbD^n\to \bbD^n\times [0,1]$ $(\bbS^{-1}=\emptyset)$ are the generating cofibrations and the generating acyclic cofibrations in $\tT$, respectively. 
	
	Moreover, if  $(\cX,\cY)$ is a pair of collections  for $M$, then the adjoint pair $$\adju{\Upsilon}{[{\cO_{(\cX,\cY)}}^{\opp},\tT]}{M\tT}{\frX}$$ is a Quillen equivalence.\end{theorem}
We call this model structure the  \emph{$(\cX,\cY)$-model structure}, and weak equivalences, fibrations and cofibrations in this model structure will be called $({\cX},{\cY})$-weak equivalences, $({\cX},{\cY})$-fibrations and $({\cX},{\cY})$-cofibrations.

This paper is organized as follows. In Section \ref{sect:preliminaries} we first give some preliminary constructions; including, the induced orbit category and the functor $\rnv{M}$. We discuss some basic properties of $\rnv{M}$ and demonstrate these properties on simple examples. We then, in \ref{sssect:topological-prop-of-rinv}, list some topological properties of $\rnv{M}$, including  Lemmas \ref{lem:rinv-preserve-pushouts} and \ref{lem:rinv-is-relatively-small}, which we directly use in the proof of Theorem \ref{thm:Rmodel}. In Section \ref{ssect:model-categorical-preliminaries} we briefly mention model categorical preliminaries; such as, transferring model structures along a right adjoint and standard equivariant model structures. In Section \ref{sect:model-structures-on-m-spaces} we prove our main results.  First we prove a preliminary form, Proposition \ref{prop:preliminarymodels}, then we describe the adjunction $\adj{\Upsilon}{\frX}$ and prove Theorem \ref{thm:Rmodel}. We then proceed in Section \ref{ssect:consequences} with some immediate consequences of  Theorem \ref{thm:Rmodel}; such as, constructions of  $M_{(\cX,\cY)}$-$CW$-complexes and $M_{(\cX,\cY)}$-Whitehead Theorems for these new model categories, see Section \ref{sssect:cwcomplexes} and Corollary \ref{cor:whitehead}. If $(\cX,\cY)$ is a pair of collections  for $M$, then further consequences of Theorem \ref{thm:Rmodel} follow, including classifying spaces for families of collections, Eilenberg-Maclane $M$-spaces and Bredon-like cohomology theories for $M$-spaces, see Sections \ref{sssect:classifying-spaces}, \ref{sssect:eilenberg-maclane-spaces} and \ref{sssect:mcohomology}. Finally in  Section \ref{sssect:remarks-on-applications} we point out relations and applications of our approach on dynamical properties of (irreversible) systems; such as,  attracting sets, periodicity, immortality and time reversal symmetries.

\section{Preliminary constructions and lemmas}\label{sect:preliminaries}
\begin{notation} We use the notation $\Grp$ for the category of groups, $\Mon$ for the category of monoids, $\Set$ for the category of sets, $\sS$ for the category of simplicial sets, $\tT$ for the category of compactly generated weak Hausdorff topological spaces and $\Cat$ for the category of small categories. For any monoid or group $D$ and a category $\frC$ we use $D\frC$ for the category of $D$-objects in $\frC$, i.e., the functor category $[D,\frC]$ when $D$ is regarded as a category with single object.
\end{notation}
As is well known, the forgetful functor $U:\Grp\to \Mon$ admits a left adjoint $G:\Mon\to \Grp$, called \emph{the group completion}. For a given monoid $M$, the group completion $G(M)$ can be constructed as follows: let $F_G(M)$ denote the free group on the elements of $M$, with the `set' inclusion $i:M\to F_G(M)$. Then $G(M)$ is the quotient group of $F_G(M)$ by the normal subgroup generated by the set of words $$\{i(m)i(n)i(m n)^{-1} \ {\mid} \ m,n \in M\}.$$ In other words, $G(M)$ is the quotient of $F_G(M)$ by all relations obtained by the multiplication in $M$, see \cite{malcev}. The universal homomorphism $q_M:M\to G(M)$ is given by the composition of $i$ with the quotient homomorphism $F_G(M)\to G(M)$. 

\subsection{Category of induced orbits}\label{ssect:induced-orbit-category}
Given a submonoid $N\leq M$, the restriction, $\Res_N^M:M\tT\to N\tT$, admits a left adjoint; namely, the induction $\Ind^M_N$, given by $$\Ind^M_N(X)=M\times_N X=M\times X/\sim$$ for an $N$-space $X$ where $ M\times_N X$ is the quotient of $M\times X$ with respect to the equivalence relation generated by $(m n,x)\sim (m,n\cdot x)$.

For a pair $(N,H)$ where $N\leq M$ is a submonoid and $H\leq G(N)$ is a subgroup, the $M$-space $M\times_N q^*_N(G(N)/H)$ is called an \emph{induced orbit} (as its the induction of a $G(N)$-orbit). Here, the underlying space of $q^*_N(G(N)/H)$ is the orbit $G(N)/H$.  We define the \emph{induced orbit category of $M$} as the full-subcategory of $M\tT$ consisting of $M$-spaces $M\times_N q^*_N(G(N)/H)$ with the $M$-action induced by translation on $M$ for every  $N\leq M$ and $H\leq G(N)$. For simplicity, we denote the objects of this category by pairs $(N,H)$.

More generally, one defines the induced orbit category  with respect to a given  pair of collections. Let $\cX$ be a set of submonoids of $M$ containing the trivial submonoid $\mathbf{e}$, for each $N\in \cX$ let $\cY_N$ be a collection of subgroups of $G(N)$ containing the trivial subgroups $\mathbf{e}$ and let $\cY$ be the set of all such collections. Here, by a \emph{collection of subgroups} we mean a subset of the set of all subgroups that is closed under conjugation. One defines $\cO_{(\cX,\cY)}$ as the full-subcategory of $\cO(M:q)$ spanned by objects $(N,H)$ with $N\in \cX$ and $H\in \cY_N$. Such a pair is called a \emph{pair of collections} if $\cO_{(\cX,\cY)}$ is replete in $\cO(M:q)$.

\subsection*{Induced orbits for commutative monoids}
The structure of induced orbits might be complicated for an arbitrary monoid, since the relation $\sim$ in the definition is usually not an equivalence relation. However, for nice commutative monoids (e.g., cancellative ones) its structure is not complicated. In fact, if $M$ is commutative, then  $M\times_N q^*_N(G(N)/H)$ is in the image of $q^*_M$ (i.e., $M$ acts  by bijections), provided that for every $t\in M$ there exists $u\in N$ and $b\in M$ such that $q_M(b)=q_M(t)^{-1}q_m(u)$ (i.e., $q_M(t)^{-1}q_m(u)$ is in the image of $q_N:N\to G(N)$). 
In this case, we can extend the $M$-action on $M\times_N q^*_N(G(N)/H)$ to a $G(M)$-action by $q(t)^{-1}\cdot [m,gH]= [mb,q_M(u)^{-1}\cdot gH]$. For example, $M=\bbN$ and $N=a\bbN$ for some integer $a\geq 1$ and $H=ba\bbZ \leq a\bbZ=G(a\bbN)$, then 
$\bbN\times_{a\bbN} q^*_{a\bbN}(a\bbZ/ba\bbZ)$ is $\bbN$-equivariantly isomorphic to $\bbZ/ba\bbZ$ with left translation, where the isomorphism is given by $$\bbN\times_{a\bbN} q^*_{a\bbN}(a\bbZ/ba\bbZ)\to \bbZ/ba\bbZ: [x,y]\mapsto x+y.$$
Note that if $N$ is trivial, then $M\times_N q^*_N(G(N)/H)$ is $M$-equivariantly isomorphic to $M$ with translation.

\subsection{Right adjoint to the restriction along $q_M$}\label{ssect:adjoints-of-q}
Let  $q_M^*: G(M)\tT  \to M\tT$ be the restriction along $q_M:M\to G(M)$, which sends a $G(M)$-space $A$ to itself with the restricted $M$-action. The universal property of $q_M$ implies that any $M$-space on which $M$ acts by isomorphisms is isomorphic to an object in the image of $q_M^*$. The functor $q_M^*$ is full and faithful, and admits a right adjoint denoted by ${\rnv{M}}$.

The description of ${\rnv{M}}$ can explicitly be given as a hom-set as follows. Consider $G(M)$ as a discrete left $M$-space with the $M$-action given by left translation, i.e., $m \cdot g=q(m)g$ for $m\in M$ and $g \in G(M)$. For an $M$-space $X$, define $\rinv{M}{X}$ as the hom-space $\Hom_M(G(M),X)$ in $M\tT$. The left $G(M)$-action on $\rinv{M}{X}$ is defined by the right translation on $G(M)$, i.e., for $\sigma \in \rinv{M}{X}$ and $g,h\in G(M)$ we have $(g\cdot \sigma) (h)= \sigma(hg)$.  It is straightforward that the natural $M$-equivariant map $\epsilon_X:q^*_M(\rinv{M}{X})\to X:\sigma\mapsto \sigma(1)$ is terminal in $M$-equivariant maps from $G(M)$-spaces to $X$.  If $X=q^*_M(\tilde{X})$ for some $G(M)$-space $\tilde{X}$, then the map $X\to q^*_M( \rinv{M}{X}): x\mapsto (\sigma_x: g \mapsto g\cdot x)$ is the inverse of $\epsilon_X$.

A subspace $U $ of $X$ is $M$-invariant if  it is invariant under the action, i.e., $m\cdot U=U$ for any $m\in M$ (here $m\cdot U=\{m\cdot u\ {\mid}\ u\in U\}$). We have the following observation.
\begin{proposition}\label{prop:imepsilonisMinvariant}
	For an $M$-space $X$, $\epsilon_X(\rinv{M}{X})$ is the maximal $M$-invariant subspace of $X$, i.e., $\epsilon_X(\rinv{M}{X})$  is $M$-invariant and contains every other $M$-invariant subspace.
\end{proposition}
\begin{proof}
	Clearly $\epsilon_X(\rinv{M}{X})$ is $M$-invariant. Assume $U\subseteq X$ is an $M$-invariant subset of $X$; i.e., $m\cdot U=U$ for any $m\in M$. Thus, for every $u\in U$ and $m\in M$,  there exists $v,w\in U$ with $m\cdot v=u$ and $m\cdot u=w$. Therefore, there exists $\sigma:G(M)\to X$ such that $\sigma(1)=u$. But this implies $U\subseteq \epsilon_X(\rinv{M}{X})$.
\end{proof}
On finite sets ${\rnv{M}}$ sends a finite $M$-set $A$ to an $M$-set that is isomorhic to its unique maximal $M$-invariant subset; that is, the maximal subset of $U\subseteq A$ satisfying $m\cdot U=U$ for every $m\in M$. Note that when restricted on the category of finite sets, ${\rnv{M}}$ is naturally isomorphic to the functor given in \cite[Sec. 4.3, Prop. 6]{erdalunlu}, provided that the action is one-sided. The natural isomorphism is induced by the map $M\to G(M):m\mapsto q(m)^{-1}$ on the contravariant-homs.

\begin{proposition}
	Let $X$ be a $M$-space and $U\subseteq X$ is an open  such that  $\exists\ m\in M$ with $\{x\in X\ {\mid}\ m\cdot x\in U\}$ (i.e., preimage of $U$ under $x\mapsto m\cdot x$) is disconnected. Then $\epsilon_X^{-1}(U)$ is disconnected.
\end{proposition}
The proof is straightforward as if $\{x\in X\ {\mid}\ m\cdot x\in U\}=U_1\cup U_2$ for disjoint open sets $U_1,U_2$, then $\{\sigma\in \rinv{M}{X}\ {\mid}\ \sigma(m^{-1})\in U_1\}$ and $\{\sigma\in \rinv{M}{X}\ {\mid}\ \sigma(m^{-1})\in U_1\}$ are disjoint open sets and their union is $\epsilon_X^{-1}(U)$.

Combining the two previous propositions, we can conclude that for an $M$-space $X$, $\rnv{M}$ provides a decomposition of the maximal $M$-invariant subspace of  $X$ into subspaces on which $M$ acts by homeomorphisms.

Consider the following example.
\begin{example}\label{ex:RUS} Let $X=(-\infty ,-1]\cup\bbS^1  \subseteq \bbR^2$ with the $\bbN$-action given by $1\cdot (x,0)=(x+1,0)$ if $x\leq -2$,  $1\cdot (x,0)=(\sin(x\pi/2),\cos(x\pi/2))$ if $x\in [-2,-1]$
	and $\bbN$ acts $(compatibly)$ by $\pi/2$ rotation to counter-clockwise direction on $\bbS^1$, as described in the following figure 
	\[\begin{tikzpicture}[scale=.75]
		\draw (-4,0) --(-1,0);
		\draw [domain=0:360] plot ({cos(\x)}, {sin(\x)});
		
		\node (c) at (0,-1.3) {};
		\node (d) at (1.3,0) {};
		\node (aA) at (-3.5,-.3) {};
		\node (bB) at (-1.5,-.3) {};
		\node (cc) at (-1.6,.2) {};
		\node (dd) at (-2,.7) {};
		
		\path[->,font=\scriptsize]
		(aA) edge node[below]  { } (bB)
		(c) edge[bend right=30] node[right]  { } (d);
	\end{tikzpicture}\]
	Then $ \rinv{\bbN}{X}\cong\bbR \amalg \bbS^1$ where $\bbZ$ acts by translation on $\bbR$ component and by $\pi/2$ rotations on $\bbS^1$ component. The component that is isomorphic to $ \bbS^1$ consist of maps $\sigma:\bbZ\to X$ such that $\sigma(n)\in  \bbS^1$ for every $n\in \bbZ$. On the other hand, the component that is isomorphic to $\bbR$ consist of maps $\sigma:\bbZ\to X$ such that $\sigma(-n)\in  (-\infty ,-1)$ for sufficiently large $n$.
	
	Now let  $V=\{(x,y)\in X\ {\mid} \ x\geq -2\}$ with the $M$-action defined in the same way (here note that $n\cdot V\subseteq V$ for every $n\in \bbN$). Then $ \rinv{\bbN}{V}\cong\bbS^1$ with the rotation action. Note that as the points on the $x$-axis cannot be reversed indefinitely, they vanish after applying $\rnv{\bbN}$.
\end{example}
Note  that restriction to submonoid $N\leq M$ and then applying $\rnv{N}$ provide different symmetry information. It can be easily demonstrated by the following example.
\begin{example}\label{ex:sub} Let $Y=\{y_0,y_1,y_2\}$ and $M=F(a,b)$, the free monoid generated by $a$ and $b$. Consider the $M$ action on $Y$ generated by $a\cdot y_0=b\cdot y_0=y_0$, $a\cdot y_1=y_2$, $b\cdot y_1=y_0$, $a\cdot y_2=y_1$ and $b\cdot y_2=y_2$; see the following figure
	\[\begin{tikzpicture}[node distance=25mm, auto, scale=.5]
		\node[state]   (sA)               {$y_0$};
		\node[state] (sB) [right of=sA] {$y_1$};
		\node[state] (sC) [right of=sB] {$y_2$};
		\path[->] (sA) edge[loop left] node {$a,b$} ()
		(sC) edge[loop right] node {$b$} ()
		(sB) edge[bend left]  node {$a$} (sC)
		(sC) edge[bend left]  node {$a$} (sB)
		(sB) edge  node {$b$} (sA);
	\end{tikzpicture}\]
	Then $\rinv{M}{Y}=\{y_0\}$ with the trivial $F_G(a,b)$-action and if $N=\langle a\rangle$, then $\rinv{N}{\Res^M_NY}=\{y_0,y_1,y_2\}$ with the $F_G(a)$-action given by $a\cdot y_1=y_2,\ a\cdot y_2=y_1$ and $a\cdot y_0=y_0$. If $K= \langle b\rangle$, then $\rinv{K}{\Res^M_KY}=\{y_0,y_2\}$ with the trivial action. If $L= \langle a^2\rangle$, then $\rinv{L}{\Res^M_LY}=\{y_0,y_1,y_2\}$ with the trivial action. 
\end{example}
The functor $\rnv{M}$, as an invariant of $M$-spaces, is stronger than the fixed points. 
\begin{example}\label{ex:RuR} Let $Z=\{(x,y): xy=0\}\subset\bbR^2$ with the $\bbN$-action given by $1\cdot (x,0)=(x+1,0)$ for every $x\in \bbR$  $1\cdot (0,y)=(0,y({\mid}y{\mid}-1)/{\mid}y{\mid})$ if ${\mid}y{\mid}>1$ and $1\cdot (0,y)=(1-{\mid}y{\mid},0)$ if ${\mid}y{\mid}\leq 1$. See the following figure 
	\[\begin{tikzpicture}[scale=1]
		\draw (-1.5,0) --(1.5,0);
		\draw (0,1) --(0,-1);
		
		\node (c) at (0.12,-1) {};
		\node (d) at (1,-0.12) {};
		\node (c1) at (0.12,1) {};
		\node (d1) at (1,0.12) {};
		
		\node (cc) at (-1.3,.1) {};
		\node (dd) at (-.1,0.1) {};
		
		\node (cc1) at (1.5,.1) {};
		\node (dd1) at (3,0.1) {};	
		\path[->,font=\scriptsize]
		(cc) edge node[below]  { } (dd)
		(c) edge[bend left=55] node[right]  { } (d)
		(c1) edge[bend right=55] node[right]  { } (d1);
	\end{tikzpicture}\]
	We have  $\rinv{\bbN}{Z}\cong \bbR\amalg \bbR\amalg \bbR$ where on each connected component $\bbN$ acts by translation. 
	For every nontrivial submonoid of $\bbN$ the associated fixed point space is empty and for trivial submonoid the fixed point space is $Z$, which is contractible as a topological space. However, $\rinv{\bbN}{Z}$ as a $\bbZ$-space is not contractible.

\end{example}

The interaction between $\rnv{M}$ and ordinary fixed points of group action can be seen as follows.
\begin{proposition}\label{prop:rinvandfixedpoints}
	If $H\leq G(M)$ and $X$ is an $M$-space, then we have a homeomorphism $\hom_M(G(M)/H,X)\cong (\rinv{M}{X})^H$.
\end{proposition}
\begin{proof} 
	The quotient map $p:G(M)\to G(M)/H$ induces an inclusion $p^*:\hom_M(G(M)/H,X)\to \hom_M(G(M),X)$. If $\sigma=p^*(s)$ for some $s:G(M)/H\to X$, then $$(h\cdot( s\circ p))(g)= s\circ p(gh)=s(ghH)=s(gH)=s\circ p(g).$$ So, the image of $p^*$ is $H$-fixed.  Now, if  $\sigma \in (\rinv{M}{X})^H$, then $s_\sigma(gH)=\sigma(g)$ is a well defined element in $\hom_M(G(M)/H,X)$ with $p^*(s_\sigma)=\sigma$. Thus, any  $H$-fixed element in $\rinv{M}{X}$ lies in the image of $p^*$. The topologies also agree since in both sides open sets determined coordinatewise. 
\end{proof}

\subsubsection{Further topological properties of $\rnv{M}$}\label{sssect:topological-prop-of-rinv} 
Here we list some topological properties of ${\rnv{M}}$. These properties, especially Lemmas \ref{lem:rinv-preserve-pushouts} and \ref{lem:rinv-is-relatively-small}, will be crucial in proving our main results.

Observe that $\rinv{M}{X}=\Hom_M(G(M),X)$ is a subspace of $\Hom(G(M),X)$. Thus, if $X$ is Hausdorff, then so is $\rinv{M}{X}$. We also have the following lemma.
\begin{lemma}\label{lem:rinv-is-closed}
	For any $M$-space $X$, $\rinv{M}{X}$ is a  closed subspace of $\Hom(G(M),X)$. In particular, if $X$ is compact, then so is  $\rinv{M}{X}$. 
\end{lemma}
\begin{proof}
	Given $m\in M$, let $\Phi_m:\Hom(G(M),X)\to \Hom(G(M),X)$ be the map given by $\Phi_m(\sigma)(g)=m\cdot \sigma (q(m)^{-1}g)$ for every $\sigma \in \Hom(G(M),X)$ and $g \in G(M)$. Since $M$ acts on $X$ by continuous maps and $G(M)$ is a discrete space and the compositions is continuous, $\Phi_m$ is continuous. Denote by $\Fix(\Phi_m)$ the set of fixed points of $\Phi_m$, which is a closed subspace of $\Hom(G(M),X)$ since it is an equalizer and equalizers are closed \cite[Cor. 2.15]{strickland2009category}. If $\sigma$ is $M$-equivariant, i.e., $\sigma\in \rinv{M}{X}$, then for every $m\in M$ and $g\in G(M)$ we have $$\Phi_m(\sigma)(g)=m\cdot \sigma (q(m)^{-1}g)=\sigma(g) .$$ In other words, $\displaystyle\rinv{M}{X}\subseteq \bigcap_{m\in M}\Fix(\Phi_m)$. Conversely, assume that  $\displaystyle\sigma\in \bigcap_{m\in M}\Fix(\Phi_m)$. Then for every $m\in M$ and $g\in G(M)$ we have $\Phi_m(\sigma)(g)=\sigma(g)$. Thus, $$\sigma(q(m)g)=\Phi_m(\sigma)(q(m)g)=m\cdot \sigma (q(m)^{-1} q(m)g)=m\cdot \sigma (g),$$ i.e., $\sigma$ is $M$-equivariant. Thus, we have $\displaystyle\rinv{M}{X}= \bigcap_{m\in M}\Fix(\Phi_m)$, which  is a closed subspace of $\Hom(G(M),X)$. 
	
	If $X$ is compact, then so is $\Hom(G(M),X)$, and thus $\Hom_M(G(M),X)$.
\end{proof}

For any space $A$ and any $M$-space $X$,  as $\tT$ is cartesian closed,  we have natural homeomorphisms 
$$\Hom(G(M), \Hom(A,X))\cong\Hom(A\times G(M), X )\cong \Hom(A,\Hom(G(M),X))$$ after forgetting the $M$-action on $X$ and $G(M)$. Consider $\Hom(A,X)$ with the pointwise $M$-action. If $f:G(M)\to \Hom(A,X)$ is an $M$-map, then  for every $a\in A$, the map $f_a:G(M)\to X$ with $f_a(g)=f(g)(a)$ is $M$-equivariant. In fact, we have $m\cdot f_a(g)=m\cdot f(g)(a)=f(q(m)g)(a)=f_a(q(m)g)$ for every $m\in M$. Hence, in the following commuting diagram 
\[\xymatrix{\displaystyle
	\Hom_M(G(M), \Hom(A,X))   \ar@{^{(}->}[d] \ar[rr]^-{}& &	\Hom(A,\Hom_M(G(M),X))\ar@{^{(}->}[d]\\
	\Hom(G(M), \Hom(A,X))  \ar[rr]^-{\cong}& &  \Hom(A,\Hom(G(M),X))}.\]
the top horizontal map is a well defined bijection.  Moreover, by Lemma \ref{lem:rinv-is-closed} vertical arrows are closed inclusions, so that the top horizontal map is closed. We obtain the following lemma.
\begin{lemma}\label{lem:powerover}
	For any space $A$ and $M$-space $X$, the natural homeomorphism $$\Hom(G(M), \Hom(A,X))\cong\Hom(A,\Hom(G(M),X))$$ restricts to a natural homeomorphism $$\Hom_M(G(M), \Hom(A,X))\cong \Hom(A,\Hom_M(G(M),X)),$$ where $\Hom(A,X)$ has the pointwise $M$-action.
\end{lemma}
Let $X$ and $Y$ be $M$-spaces and let $Z=X\amalg Y$. Let $\sigma:G(M)\to Z$ be an $M$-map. Then for every $m\in M$ the elements $\sigma(q(m)^{-1})$, $\sigma(1)= m\cdot \sigma(q(m)^{-1})$ and $\sigma(q(m))=m^2\cdot \sigma(q(m)^{-1})$ belong to the same component of $Z$.  Since every element in $G(M)$ can be written as a product of a sequence of elements in the set $\{q(m),\ q(m)^{-1}\ {\mid} \ m\in M\}$,  this implies $\sigma$ factors through either $X$ or $Y$. This shows that  ${\rnv{M}}$ preserves coproducts of $M$-sets. However, it is also straightforward that the canonical map $\rinv{M}{X}\amalg \rinv{M}{Y}\to \rinv{M}{X\amalg Y}$ is open, since any open subset of $X$ or $Y$ is also open in $X\amalg Y$. Thus, we have the following lemma.
\begin{lemma}\label{lem:invariant}
	The functor ${\rnv{M}}$ preserves coproducts in $M\tT$.
\end{lemma}
The lemma above, in particular, implies that if $X$ and $Z-X$ are  $M$-invariant subspaces of $ Z$, then any $M$-map $\sigma:G(M)\to Z$ factors through either $X$  or $Z-X$. We particularly use that  ${\rnv{M}}$ preserves coproducts of $M$-sets.

Note that closed inclusions are regular monomorphisms in $\tT$, \cite[Thm. 3.1]{strickland2009category}. Being a right adjoint,  ${\rnv{M}}$ (which is given by the hom-object $\Hom_M(G(M),-)$ in $M\tT$) sends $M$-equivariant closed inclusions to $G(M)$-equivariant closed inclusions.

\begin{lemma}\label{lem:rinv-preserve-pushouts}
	Let  $X$ be a discrete $M$-space. Then ${\rnv{M}}=\Hom_M(G(M),-)$ preserves pushouts in which one leg is of the form $\imath=id_X\times i_n$, where   $i_n:\bbS^{n-1}\into \bbD^{n}$ is the inclusion of boundary.
\end{lemma}
\begin{proof}
	Suppose that we have the following pushout diagram in $M\tT$ $$\xymatrix{\displaystyle
		X\times \bbS^{n-1}  \ar[d]_-{\imath} \ar[rr]^-{f}& &	Y\ar[d]^-{\widetilde{\imath }}\\
		X\times  \bbD^{n}  \ar[rr]_-{\widetilde{f }}& &  P}.$$
	Applying ${\rnv{M}}$ to the diagram, we get that $\rinv{M}{P}$ is a cocone for the diagram defined by $\rinv{M}{f}$ and $\rinv{M}{\imath}$. Let $\widetilde{P}$ be the pushout of $\rinv{M}{f}$ and $\rinv{M}{\imath}$, and $\kappa$ be the canonical $G(M)$-map of the pushouts as shown in the following diagram	$$\xymatrix{\displaystyle
		\rinv{M}{X\times  \bbS^{n-1}}\ar[d]_-{\rinv{M}{{\imath}}}\ar[rr]^-{\rinv{M}{f}} & &  \rinv{M}{Y}\ar@/^2.0pc/@[black][rdd]^-{\rinv{M}{\ \widetilde{\imath}\ }}\ar[d]^-{\widetilde{\rinv{M}{ \imath }}}&\\
		\rinv{M}{X\times  \bbD^{n}}\ar@/^{-2.0pc}/@[black][rrrd]_-{\rinv{M}{\ \widetilde{f}\ }}  \ar[rr]_-{\widetilde{\rinv{M}{f}}} & &  \widetilde{P}\ar[dr]^-{\kappa}&\\
		&&& \rinv{M}{P}}$$
	We show that $\kappa$ is a homeomorphism.
	
	Since $\imath$ is a closed inclusion, the pushouts are created in the category of sets, see \cite[Prop. 2.35]{strickland2009category}.  Thus, following the construction in \cite[Prop. 2.35]{strickland2009category}, we can set $$P= X \times(  \bbD^{n}- \bbS^{n-1})  \amalg   Y$$ with the topology so that the map $$p: X \times  \bbD^{n} \amalg   Y \to P$$ that is identity on $ P \subset X \times  \bbD^{n} \amalg   Y $ and $f$ on $X\times  \bbS^{n-1} \subset X \times  \bbD^{n} \amalg   Y$ is a quotient map, and $\widetilde{f }$ and $\widetilde{\imath }$ are restrictions of $p$. The action is defined on each component. Since ${\rnv{M}}$ preserves closed inclusions, $\rinv{M}{{\imath}}$ is also a closed inclusion. As above, we can set  $$\widetilde{P}=\rinv{M}{Y} \amalg \rinv{M}{X\times (\bbD^{n}-\bbS^{n-1})},$$ so that the map $$\widetilde{p}:\rinv{M}{Y} \amalg \rinv{M}{X\times\bbD^{n}}\to \widetilde{P}$$ that is identity on $ \widetilde{P}$ and $\rinv{M}{f}$ on $\rinv{M}{X\times\bbS^{n-1}},$ is a quotient map. The maps $\widetilde{\rinv{M}{ \imath }}$ and $\widetilde{\rinv{M}{ f }}$ are the obvious restrictions. By Lemma \ref{lem:invariant}, $\kappa$ is a bijection and the restrictions  $\kappa{\mid}_{\rinv{M}{X\times (\bbD^{n}-\bbS^{n-1})}}$ and $\kappa{\mid}_{\rinv{M}{Y}}$ are embeddings, where $\kappa{\mid}_{\rinv{M}{Y}}$ is closed.

	Let $V\subseteq \rinv{M}{P}$ be compact. For each $g\in G(M)$,  $\tau_{g}:\rinv{M}{P}\to P:\sigma\mapsto \sigma(g)$ is continuous, so that $\tau_g(V) =: U_g\subseteq P$ is compact. Observe that $(P,Y)$ is a relative cell-complex in $\tT$ (after forgetting the actions). Then, $\widetilde{f}^{-1}(U_g)$ intersect only finitely many  open cells in $X \times(\bbD^{n}- \bbS^{n-1})$,  as otherwise  we can find an open cover for $U_g$ without a finite subcover (see also \cite[Prop. 4.10]{hirschhorn2019quillen}). Thus, $\widetilde{f}^{-1}(U_g)$ can be expressed as a union of a compact subset $K_g$ of $X \times \bbD^{n} $ and a subset $C_g$ of  $\imath(X \times \bbS^{n-1})$. Here $K_g$ is chosen as the disjoint union of all closed discs whose interior intersects with $\widetilde{f}^{-1}(U_g)$, which is compact as it is disjoint union of finitely many closed discs. Let $K=(\prod_{g\in G(M)}K_g)\cap 	\rinv{M}{X\times  \bbD^{n}}$ where both spaces in the intersection are regarded as subsets of $\hom(G(M),X\times \bbD^n)$. By Lemma \ref{lem:rinv-is-closed} $K$ is a compact subset of $\rinv{M}{X\times  \bbD^{n}}$. Let $C=\rinv{M}{\ \widetilde{f}\ }^{-1}(V)-K$. Then $C$ is contained in the image of $\rinv{M}{\imath}$. In fact, if $\sigma\in C$, then for some $g\in G(M)$, $\sigma(g)\in \imath(X \times \bbS^{n-1})$, and since the action is defined only on the first coordinate, $\sigma(g)\in \imath(X \times \bbS^{n-1})$ for every $g\in G(M)$. Therefore, $\sigma $ belongs to the image of $\rinv{M}{\imath}$. 
	Since both $\widetilde{\rinv{M}{\imath}}$ and $\rinv{M}{\ \widetilde{\imath}\ }$ are closed inclusions,  $\widetilde{\rinv{M}{\imath}}(\rinv{M}{\ \widetilde{\imath}\ }^{-1}(V))=\kappa^{-1}(V)\cap \widetilde{\rinv{M}{\imath}}(\rinv{M}{Y})$ is compact. Due to the commutativity of above diagram and ${\rinv{M}{\imath}}$ being an inclusion, we have $\widetilde{\rinv{M}{f}}(C)\subseteq \kappa^{-1}(V)\cap \widetilde{\rinv{M}{\imath}}(\rinv{M}{Y})$. Since $K$ is compact, so is $\widetilde{\rinv{M}{f}}(K)$. Moreover, $\kappa^{-1}(V)=(\kappa^{-1}(V)\cap \widetilde{\rinv{M}{\imath}}(\rinv{M}{Y}))\cup\widetilde{\rinv{M}{f}}(K)$. Therefore, $\kappa^{-1}(V)$ is compact, i.e.,  $\kappa$ is proper. A proper continuous bijection in $\tT$ is a homeomorphism (see   \cite[Prop. 3.17]{strickland2009category} and its proof).
\end{proof}

The functor ${\rnv{M}}$ does not preserve sequential colimits in general, not even in the case when the underlying maps are inclusions. For example, let $M=\bbN$ and for each $n\in \bbN$ let $[\mathbf n]=\{k\in \bbN \ {\mid} \  k\leq n\}$ be the set with $\bbN$-action given by $1\cdot k=k-1$ for each $k\in [\mathbf n]-\{0\}$ and $1\cdot 0=0$. The colimit of these inclusions $[\mathbf n] \into [\mathbf{n+1}]$ is $\bbN$ with $\bbN$-action  $1\cdot k=k-1$ for all $k\in \bbN-\{0\}$ and $1\cdot 0=0$. We have $\rinv{\bbN}{[\mathbf n]}=\{0\}$ with trivial $\bbZ$-action for every $n$. Hence $\lim_{n}\rinv{\bbN}{ [\mathbf n]}=\{0\}$ with trivial $\bbZ$-action, but $\rinv{\bbN}{ \lim_{n}[\mathbf n]}=\{0\}\amalg \bbZ$ with translation on $\bbZ$ component trivial action on $\{0\}$. However, for every pair $({\cX},{\cY})$ the inclusions $[\mathbf n] \into [\mathbf{n+1}]$ are not $({\cX},{\cY})$-cofibrations.  In fact, for the inclusion $[\mathbf 1] \into [\mathbf{2}]$ the left lifting property against $[0,1]\to *$ fails, where  $[0,1]$ is considered with the $\bbN$-action given by $1\cdot x=0$ for every $x\in [0,1]$. Note that $[0,1]\to *$ is an acyclic $({\cX},{\cY})$-fibration for every pair $({\cX},{\cY})$. If a lift $h$ exists in the following commuting diagram 
$$\xymatrix{
	[\mathbf{1}]\ar[d]_{i} \ar[r]^{g} &	[0,1]\ar[d]^{}\\
	[\mathbf{2}]\ar[r]^{}\ar[ru]^{h} & \ast},$$
such that $i$ and $g$ are inclusions of subspaces, then $h(i(1))=h(1)=g(1)=1$. But then  $1\cdot h(2)=1$, which is not possible in $[0,1]$.

On the other hand, ${\rnv{M}}$ preserves sequential colimits that are relevant to constructions of cellular extensions for generating maps in $M\ii_{(\cX,\cY)}$ and $M\jj_{(\cX,\cY)}$ of Theorem \ref{thm:Rmodel}. We have the following lemma. 
\begin{lemma}\label{lem:rinv-is-relatively-small}
	Let $M$ be a   monoid and $(\cX,\cY)$  be a pair of collections. Then, given an infinite regular cardinal $\lambda$, ${\rnv{M}}$ preserves sequential colimits of $\lambda$-sequences in which underlying maps are pushouts of maps in $M\ii_{(\cX,\cY)}$ or $M\jj_{(\cX,\cY)}$.
\end{lemma}
\begin{proof}
	Let $\lambda$ be an infinite regular cardinal. Let $X:\lambda\to M\tT$ be a $\lambda$-sequence where for each  $\alpha<\lambda$ we have the following pushout square
	$$\xymatrix{\displaystyle
		M\times_{N} q^*_N(G(N)/H) \times \bbS^{n-1}  \ar[d]_{} \ar[r]^{} &		X_{\alpha}\ar[d]^{\jmath_\alpha}\\
		\displaystyle	M\times_{N} q^*_N(G(N)/H) \times \bbD^{n}  \ar[r]^{} & 	X_{\alpha+1}}.$$
	Note that for every $\alpha$ the map $\jmath_\alpha$ is a closed inclusion, implying that the $M$-map $\imath_\alpha:X_{\alpha}\to\displaystyle\colim_{\alpha<\lambda}X_{\alpha}$ is a closed inclusion, e.g., see \cite[Lem. 3.3]{strickland2009category}. Since $\rinv{M}{\displaystyle\colim_{\alpha<\lambda}X_{\alpha}}$ is a cocone for $\rinv{M}{X}$, we have a canonical $G(M)$-map of the colimits  $$\kappa:\displaystyle\colim_{\alpha<\lambda}\rinv{M}{X_{\alpha}}\to\rinv{M}{\displaystyle\colim_{\alpha<\lambda}X_{\alpha}}.$$ We show that $\kappa$ is a homeomorphism.

	We first show that $\kappa$ is a  bijection. It is enough to show that every $M$-map $\varsigma:G(M)\to\displaystyle\colim_{\alpha<\lambda}X_{\alpha}$ factors through the inclusion  $\imath_\beta:X_{\beta}\into\displaystyle\colim_{\alpha<\lambda}X_{\alpha}$ for some $\beta < \lambda$. Assume this is not the case. Then  there exist elements $g$ and $h$ in $G(M)$ and ordinals $\alpha < \beta < \lambda$ such that $\varsigma(g)\in \imath_\alpha(X_{\alpha})$ and $\varsigma(h)\in\imath_\beta(X_{\beta})-\imath_\alpha(X_{\alpha}).$ By the first paragraph of the proof of Lemma \ref{lem:rinv-preserve-pushouts}, we have $\imath_{\alpha+1}(X_{\alpha+1})-\imath_\alpha(X_{\alpha})$ is invariant under the $M$-action on $\imath_{\alpha+1}(X_{\alpha+1})$. By transfinite induction, the difference $\imath_\beta(X_{\beta})-\imath_\alpha(X_{\alpha})$  is invariant under the $M$-action on $\imath_\beta(X_{\beta})$. Since $\imath_\alpha(X_{\alpha})$  is also invariant under the  $M$-action, by Lemma \ref{lem:invariant}  we get that $\varsigma$ factors through either $\imath_\alpha(X_{\alpha})$ or $\imath_\beta(X_{\beta})-\imath_\alpha(X_{\alpha})$. This is a contradiction, and therefore, $\kappa$ is a bijection.

	Let  $A$ be a compact Hausdorff space and let $\Hom(A,X_{\alpha})$ and $\Hom(A,\displaystyle\colim_{\alpha<\lambda}X_{\alpha})$ be $M$-spaces with pointwise actions. We have the following commuting diagram
	$$\xymatrix{
		\displaystyle\colim_{\alpha<\lambda}[\Hom(A,\rinv{M}{X_{\alpha}})]  \ar[d]_{f} \ar[rr]^-{t} &&\Hom(A,\displaystyle\colim_{\alpha<\lambda}(\rinv{M}{X_{\alpha}}) \ar[d]^{\Hom(A,\kappa)} \\ 	\displaystyle\colim_{\alpha<\lambda}[\rinv{M}{\Hom(A,(X_{\alpha}))}]\ar[d]_{e}&&\Hom(A,\rinv{M}{\displaystyle\colim_{\alpha<\lambda}X_{\alpha}})\ar[d]^{d} \\
		\rinv{M}{\displaystyle\colim_{\alpha<\lambda}\Hom(A,(X_{\alpha}))}  \ar[rr]^-{s}&& \rinv{M}{\Hom(A,\displaystyle\colim_{\alpha<\lambda}X_{\alpha})}}.$$
	Here, $t$ and $e$ are obvious canonical maps of colimits,  $s$ is the map induced by the canonical map of the colimits, and the maps $d$ and $f$ are obtained from the map given in Lemma \ref{lem:powerover}.  Since $\Hom(A,-)$ is a right adjoint, $\Hom(A,\kappa)$ is injective. By Lemma \ref{lem:powerover} the maps $d$ and $f$ are homeomorphisms.  Being a right adjoint, ${\rnv{M}}$  preserves regular monomorphisms, so that $\rinv{M}{X}$ is a sequence of closed inclusions. By \cite[Lem. 3.8]{strickland2009category}, $\Hom(A,-)$ preserves colimits of $\lambda$-sequences whose underlying maps are closed inclusions. Thus, $t$ is a homeomorphism. The canonical map $\displaystyle\colim_{\alpha<\lambda}\Hom(A,X_{\alpha})\to \Hom(A,\displaystyle\colim_{\alpha<\lambda}X_{\alpha})$ is an $M$-equivariant bijection and since $G(M)$ is discrete ${\rnv{M}}=\hrnv{M}{-}$ preserves $M$-equivariant bijections, so that $s$ is a bijection. By the previous paragraph, combined with the fact that $\Hom(A,X)$ is a $\lambda$-sequence of closed inclusions, $e$ is a bijection.  This implies $\Hom(A,\kappa)$ is a bijection. Since this is true for every compact Hausdorff space $A$ and in $\tT$ every space is compactly generated, $\kappa$ is an homeomorphism. The proof for sequences of relative $M\jj_{(\cX,\cY)}$-cell complexes is similar.
\end{proof}

\subsection{Model categorical preliminaries}\label{ssect:model-categorical-preliminaries}
A model structure on a category is a choice of three distinguished classes of morphisms called weak equivalences, cofibrations and fibrations, which satisfy certain axioms that are shaped in a way to allow us `to do homotopy theory' similar to the classical homotopy theory of spaces. Categories with model structures have wide variety of examples and applications, from various different areas of mathematics. Original definition of a model category is due to Quillen \cite{quillen}.

In this paper, we conventionally assume that a model category is a bicomplete category with a model structure. Note that the original definitions of Quillen only requires having finite limits and finite colimits. Morphisms that are simultaneously fibrations (resp. cofibrations) and weak equivalences are called acyclic fibrations (resp. acyclic cofibrations). An object $X$ is fibrant if the unique morphism from $X$ to the terminal object is a fibration. Dually, $X$ is cofibrant if the unique morphism from the initial object to $X$ is a cofibration. 
We refer to \cite{dwyer}, \cite{hovey} and \cite{hirschhorn} for various properties, examples and applications of model categories. More specifically, for details on cofibrant generation we refer to  \cite[Ch. 15]{mayponto} and \cite[Ch. 11]{hirschhorn}, and for accessible model categories we refer to  \cite[Sec. 3.1 and Def. 3.1.3, 3.1.6]{riehlinduced}. 

\subsubsection{Transfer arguments for model structures}\label{sssect:transfer-arguments}
Suppose that we are given a pair of adjoint functors $\adju{F}{\frD}{\frC}{G}$ between bicomplete categories $\frC$ and $\frD$, such that $\frD$ is a cofibrantly generated model category with the generating cofibrations $\ii$ and the generating acyclic cofibrations $\jj$. Declare $f$ to be a weak equivalence (resp. fibration) in $\frC$ if $G(f)$ is a weak equivalence (resp. fibration) in $\frD$. If we let $F\ii$ and $F\jj$ are images of $\ii$ and $\jj$ under $F$, then a model structure on $\frC$ exists if
\begin{enumerate}
	\item[\textbf{(i)}]  $G$ takes relative $F\jj$-cell complexes to weak equivalences.
	\item[\textbf{(ii)}] both $F\ii$ and $F\jj$ permit the small object argument.
\end{enumerate}
This model structure is also cofibrantly generated  with the generating cofibrations $F\ii$ and the generating acyclic cofibrations $F\jj$, see \cite[Sec. 11.3]{hirschhorn}. The condition \textbf{(i)} holds if  $\frC$ has a fibrant replacement functor and has functorial path objects for fibrant objects and \textbf{(ii)} holds if $F$  preserves small objects, see \cite[Sec. 2.5 and 2.6]{moerdijkberger}. Here, an object $x$ in a category $\frC$ is \emph{small} means it is $\kappa$-small for some regular cardinal $\kappa$, i.e., the functor ${\frC}(x,-)$ preserves $\kappa$-filtered colimits, see \cite{adamek}.  In fact for \textbf{(ii)} to hold it is enough that $F$ preserves smallness relative to $F\ii$ and $F\jj$; equivalently, for a given regular cardinal $\lambda$, $G$ preserves sequential colimits of $\lambda$-sequences in which underlying maps are pushouts of maps in $F\ii$ (resp. $F\jj$), see \cite[Ch. 15.1]{mayponto}. 

\subsubsection{Standard model structures for group actions}\label{sssect:standard-equivariant-model}
The equivariant homotopy theory of group actions has been highly developed since 1970's and provides lots of tools for the systematic study of group actions. It has been crucial in the proof of the Arf-Kervaire invariant one problem \cite{hill2016nonexistence}. Equivariant homotopy theory is still instrumental in  many areas such as algebraic $K$-theory \cite{nikolaus2018topological,dotto2020real} and Mathematical Physics  \cite{huerta2019real,sati2020equivariant}. For general references on the subject we refer to  \cite{elmendorf}, \cite{piacenza}, \cite{may}, and to \cite{stephan,guillou2010enriched} for generalizations. 

Let $G$ be a group. For a collection $\cY$ of subgroups of $G$ (i.e., a set of subgroups closed under conjugation), there exists a cofibrantly generated model structure; namely, the $\cY$-model structure, on the category $G\tT$. The weak equivalences (resp. fibrations) in this model structure are $\cY$-equivalences (resp. $\cY$-fibrations) where  $f:X\to Y$ in $G\tT$ is a \emph{$\cY$-equivalence} (resp. a \emph{$\cY$-fibration})  if for every $H$ in $\cY$ $ f^H: X ^H\rightarrow Y ^H$ is a weak homotopy equivalence (resp. Serre fibration). Besides, due to Elmendorf's theorem \cite{elmendorf}, there is a Quillen equivalence $$\adju{\Phi}{[{\cO_{\cY}}^{\opp},\tT]}{G\tT}{\Psi}$$ where $\cO_{\cY}$ is the orbit category for $\cY$ and $[{\cO_{\cY}}^{\opp},\tT]$ has the projective model structure.

This model structure and the Elmendorf's theorem is generalized in  \cite{stephan} to other model categories (including $\sS$) for which the fixed point functors satisfy a condition (called the cellularity condition). In all cases, the generating cofibrations and the generating acyclic cofibrations are given by $G \ii_\cY=\{G/H\times i \ {\mid} \  n\in \bbN, H\in \cY, i\in \ii \}$ 	and $G \jj_\cY=\{G/H\times j \ {\mid} \  n\in \bbN, H\in \cY, j\in \jj\} ,$ respectively, where $\ii$ and $\jj$ are the generating cofibrations and the generating acyclic cofibrations of the underlying category.

\section{Proofs of main results and consequences}\label{sect:model-structures-on-m-spaces}
The following  is a preliminary version of Theorem \ref{thm:Rmodel}.
\begin{proposition}\label{prop:preliminarymodels}
	Let $N\leq M$ and $\cY_N$ be a collection of subgroups of $G(N)$. Then,
	
	{\bf I.} \ $N\tT$ admits a cofibrantly generated model structure in which weak equivalences and fibrations are created by ${\rnv{N}}:N\tT\to G(N)\tT$ from the $\cY_N$-model structure on $G(N)\tT$. The sets of generating cofibrations and acyclic cofibrations are $N\ii_{\cY_N}=\{q^*_N(G(N)/H)\times i_n\ {\mid}\ H\in \cY_N, n\in N\}$ and $N\jj_{\cY_N}=\{q^*_N(G(N)/H)\times j_n\ {\mid}\ H\in \cY_N, n\in N\},$ respectively. Moreover, the pair $\adj{q^*_N}{{\rnv{N}}}$ is a Quillen equivalence. 
	
	{\bf II.} \ $M\tT$ admits a cofibrantly generated model structure in which weak equivalences and fibrations are created by $\Res^M_N:M\tT\to N\tT$ from the model structure  on $N\tT$ given in {\bf I}.  The sets of generating cofibrations and acyclic cofibrations are  $M\ii_{\cY_N}=\{M\times_N q^*_N(G(N)/H)\times i_n\ {\mid}\ H\in \cY_N, n\in N\}$ and $M\jj_{\cY_N}=\{M\times_N q^*_N(G(N)/H)\times j_n\ {\mid}\ H\in \cY_N, n\in N\}$, respectively.

\end{proposition}

\begin{proof} For {\bf I} note that every object is already fibrant in $N\tT$, i.e., the identity is the fibrant replacement functor. For an $N$-space $X$ the  space $\Hom([0,1],X)$ with the pointwise action is a good path object in $N\tT$, i.e., $N\tT$ has functorial path objects. Hence, ${\rnv{N}}$ takes relative $N\jj_{\cY_N}$-complexes to weak equivalences.   Note that  $\rinv{N}{q^*_N(G(N)/H)}\cong G(N)/H$ for $H\leq G(N)$. For any regular cardinal $\lambda$, by Lemma  \ref{lem:rinv-is-relatively-small},  ${\rnv{N}}$  commutes with filtered colimits of $\lambda$-sequences in which underlying maps are relative $N\ii_{\cY_N}$ or $N\jj_{\cY_N}$-complexes.  It is known that for each $H\in \cY_N$ the $H$-fixed point functor commutes with  such sequential colimits, so both  $N\ii_{\cY_N}$ and  $N\jj_{\cY_N}$ permit the small object argument. This proves that the model structure can be transferred along $\rnv{N}$ and $\adj{q^*_N}{{\rnv{N}}}$ is a Quillen pair. Besides, for any $G(N)$-space $Y$, we have a $G(N)$-equivariant homeomorphism $ \hom_N(G(N),q^*_N(Y))\cong \hom_{G(N)}(G(N),q^*_N(Y))$ since any $N$-map between spaces with reversible $N$-actions is $G(N)$-equivariant. Therefore, the unit $Y\to \rinv{N}{q_N^*(Y)}$ is an isomorphism in $G(M)\tT$, so that the pair $\adj{q^*_N}{{\rnv{N}}}$ is a Quillen equivalence (see dual form of \cite[Lem. 1]{erdalguclukanilhan} or \cite[1.3.16]{hovey}).
	
	For {\bf II}, note that identity is again the fibrant replacement functor and the path space functor is obtained similarly. Besides, $\Res^M_N$ is also a left adjoint, so it preserves all colimits including the filtered ones. Therefore, the transfer argument applies to the adjoint pair $\adj{\Ind^M_N}{\Res^M_N}$ and we have a model structure on $M\tT$ induced by $\Res^M_N$. Since $\Ind^M_N(X)=M\times_N X$, the generating sets $M\ii_{\cY_N}$ and $M\jj_{\cY_N}$ are obtained as above.
\end{proof}

Let $(\cX,\cY)$ be a pair of collections (see Section \ref{ssect:induced-orbit-category} for the definition of a pair of collections). Define $\frX
: M\tT\to [{\cO_{(\cX,\cY)}}^{\opp},\tT]$ as follows: for $(N,H)$ in $\cO_{(\cX,\cY)}$ and $X$ in $ M\tT$ let $$\frX(X)(N,H)= \rinv{N}{\Res^M_N(X})^H\cong \Hom_M(M\times_Nq^*_N(G(N)/H),X)$$ and for every map $s$ in ${\cO_{(\cX,\cY)}}$ let $\frX(X)(s)$ be the induced map on contravariant-hom. A map $f:X\to Y$ in $ M\tT$ is sent to the obvious natural transformation of covariant-homs. This functor admits a left adjoint $$\Upsilon: [{\cO_{(\cX,\cY)}}^{\opp},\tT]\to M\tT : F\mapsto F((\mathbf{e},\mathbf{e})).$$ Since ${\cO_{(\cX,\cY)}}^{\opp}((\mathbf{e},\mathbf{e}),(\mathbf{e},\mathbf{e}))\cong \rinv{\mathbf{e}}{\Res^M_\mathbf{e} (M)^\mathbf{e}}\cong M,$
the $M$-action on $ F(\mathbf{e},\mathbf{e})$ is defined by compositions of these maps. The functor $\Upsilon$ sends a natural transformation $t:F\Rightarrow F'$ to the obvious induced map. Let $\varepsilon:\Upsilon \frX
\Rightarrow id$ be the natural transformation whose components are $\rinv{\mathbf{e}}{\Res^M_\mathbf{e}(X)}^\mathbf{e}\to X$. Define $\eta :id\Rightarrow \frX
\Upsilon $ so that $$\eta_F{(N,H)}:F (N,H)\to \frX
\Upsilon(F)(N,H)=\frX
( F((\mathbf{e},\mathbf{e})))(N,H)$$ is the map induced by $F(p):F(N,H)\to F(\mathbf{e},\mathbf{e})$ where $p:M\times_\mathbf{e} q^*_{\mathbf{e}}(G(\mathbf{e})/\mathbf{e}) \to M\times_N q^*_N(G(N)/H)$ is the map induced by $\mathbf{e}\into N$, i.e., the map corresponding to the class of $[\mathbf{e},\mathbf{e}H]$ under the canonical isomorphism $\Hom_M(M\times_\mathbf{e} q^*_{\mathbf{e}}(G(\mathbf{e})/\mathbf{e})  , M\times_N q^*_N(G(N)/H))\cong M\times_N q^*_N(G(N)/H) .$ One checks that $\eta$ is the unit and $\varepsilon$ is the counit of this adjunction. Note that $\frX
$ is the right Kan extension along the functor $\jmath:M\to \cO_{(\cX,\cY)}^{\opp}$ sending the single object of $M$ to $(\mathbf{e},\mathbf{e})$ and a morphism to the corresponding map with respect to the canonical isomorphism ${\cO_{(\cX,\cY)}}^{\opp}((\mathbf{e},\mathbf{e}),(\mathbf{e},\mathbf{e})) \cong M.$

Now we return Theorem \ref{thm:Rmodel} and its proof. 
\begin{proof}[Proof of Theorem \ref{thm:Rmodel}]
	Denote by $M\tT_{\cY_N}$ the model category given in part {\bf II} of Proposition \ref{prop:preliminarymodels}. Let $\Delta_\cX:M\tT\to \prod_{N\in \cX}M\tT_{\cY_N}$ be the diagonal functor. Note that $\Delta_\cX$ has a left adjoint given by taking coproduct over $\cX$. Observe that,  every object is fibrant; so we have a fibrant replacement functor, and  $\Hom([0,1],X)$ with the pointwise action is a good path object in $M\tT$ with these fibrations and weak equivalences, i.e., $M\tT$ has functorial path objects. The functor $\Delta_\cX$ admits both of its adjoints, so it preserves all colimits including the filtered ones. Hence, the transfer argument discussed in \ref{sssect:transfer-arguments} applies to $\Delta_\cX$, i.e., we can transfer a model structure on $M\tT$ from the product model structure on $\prod_{N\in \cX}M\tT_{\cY_N}$ via $\Delta_\cX$.  Observe that for every $f$ in $M\tT$, $\frX(f)$ is a weak equivalence (resp. fibration) in  the projective model structure on
	$[{\cO_{(\cX,\cY)}}^{\opp},\tT]$ if and only if  $\Delta_\cX(f)$ is a weak equivalence (resp. fibration) in the product model structure on $\prod_{N\in \cX}M\tT_{\cY_N}$. Thus, $M\tT$ admits a model structure in which weak equivalences and fibrations are created by $\frX$ from the projective model structure on
	$[{\cO_{(\cX,\cY)}}^{\opp},\tT]$. Moreover, the pair $\adj{\Upsilon}{\frX}$ is a Quillen pair.

	The projective model structure on $[{\cO_{(\cX,\cY)}}^{\opp},\tT]$ is cofibrantly generated, e.g., \cite[Lem. A.2.8.3]{lurie}. The set of generating cofibrations are $$\ii_{{\cO_{(\cX,\cY)}}}=\{{\cO_{(\cX,\cY)}}(-,(N,H))\times i_n  \ {\mid} \  n \in \bbN \}.$$  Let $F$ be a  $\ii_{{\cO_{(\cX,\cY)}}}$-cell complex, so that there exist a  regular cardinal $\lambda$ and a $\lambda$-sequence of maps $f_\alpha:F_{\alpha}\to F_{\alpha+1}$ for every $\alpha < \lambda $ with $F_{-1}=\emptyset$ fitting into the pushout diagrams of the form 
	$$\xymatrix{
		{\cO_{(\cX,\cY)}}(-,(N,H))\times \bbS^{n-1} \ar[d]_{Id\times i_n} \ar[r]^{ } &		F_{\alpha}\ar[d]^{f_\alpha}\\
		{\cO_{(\cX,\cY)}}(-,(N,H))\times  \bbD^{n}   \ar[r]^{} & 	F_{\alpha+1}}.$$
	where the colimit of $F_n$'s is $F$. For any space $A$, we have 
	\begin{align*} 
		\Upsilon({\cO_{(\cX,\cY)}}(-,(N,H))\times A) 
		&= {\cO_{(\cX,\cY)}}((\mathbf{e},\mathbf{e}),(N,H))\times A\\ 
		& \cong M\times_N q^*_N(G(N)/H) \times A.
	\end{align*} 
	Moreover, for any $(K,Q)$ we have
	\begin{align*} 
		{\cO_{(\cX,\cY)}}((K,Q),(N,H))\times A 
		&\cong (\rinv{K}{\Res^M_K(M\times_N q^*_N(G(N)/H))})^Q\times A \\ 
		&\cong (\rinv{K}{\Res^M_K(M\times_N q^*_N(G(N)/H))}\times A)^Q,
	\end{align*} 
	since the space $A$ can be considered with trivial action and ${\rnv{K}}$ and $(-)^Q$ are identity on such spaces.

	By Lemma \ref{lem:rinv-preserve-pushouts}, ${\rnv{M}}$ preserves pushouts of diagrams in which one leg is a generating cofibration. Same is true for  the restrictions $\Res^M_N(-)$ and the fixed point functors $(-)^H$ (see \cite{fausk} for the latter, which was originally due to G. Lewis). By Lemma \ref{lem:rinv-is-relatively-small}, for each $N\leq \cX$ and $H\in \cY_N$ the composition $(\rinv{N}{\Res^M_N(-)})^H$ preserves sequential colimits in which underlying maps are cofibrations. Clearly, $\frX$ preserves the initial object. Thus, $ \frX\Upsilon F$ is the direct limit over the transfinite composition of $\frX\Upsilon f_\alpha$'s obtained by pushout diagrams:
	$$\xymatrix{
		\frX\Upsilon {\cO_{(\cX,\cY)}}(-,(N,H))\times \bbS^{n-1} \ar[d]_{Id\times 	\frX\Upsilon i_n} \ar[r]^{ } &			\frX\Upsilon F_{\alpha}\ar[d]^{	\frX\Upsilon f_\alpha}\\
		\frX\Upsilon {\cO_{(\cX,\cY)}}(-,(N,H))\times  \bbD^{n}   \ar[r]^{} & 		\frX\Upsilon F_{\alpha+1}}.$$
	We have $\eta_\emptyset$ is an isomorphism. By transfinite induction, we obtain $\eta_F$ is an isomorphism. Every projectively cofibrant diagram is a retract of a $\ii_{{\cO_{(\cX,\cY)}}}$-cell complex, hence $\eta_F$ is an isomorphism for every projectively cofibrant diagram $F$.	 Since weak equivalences and fibrations of $M\tT$ are created by $\frX$, the pair $\adj{\Upsilon}{\frX}$ is a Quillen equivalence.
\end{proof} 
The model category $M\tT$ of Theorem \ref{thm:Rmodel} is already right proper since every object is fibrant. The model category $\prod_{N\in \cX}G(N)\tT$ is proper and topological since each factor in the product is so, see \cite[III Theorem 1.8]{mandell} and \cite[Appendix A-2, Proposition 2.5 and 2.6]{schwede}. The category $M\tT$ is powered and copowered over $\tT$ with powering defined via hom with pointwise $M$-action and copowering is defined as the cartesian product after regarding the topological space with the trivial $M$-action, i.e., if $A$ is in $\tT$ and $X$ is in $M\tT$ then the powering of $X$ by $A$ is the $M$-space $\hom(A,X)$ with pointwise action and the copowering of $X$ by $A$ is the product $A\times X$ with the $M$-action defined on the second coordinate. Note also that the powering in $G(M)\tT$ is also defined by hom and $\hom_M(G(M) ,{\hom(A,X)})\cong  \hom(A,\hom_M(G(M) ,{X}))$, see Lemma \ref{lem:powerover}. Thus, $\rnv{N}$ preserves powering for every $N\in \cX$. Being a right adjoint, $\rnv{N}$ preserves pullbacks, so the pull-back power axiom holds in $M\tT$ with this model structure (as it holds in $G(N)\tT$); and thus, the model category $M\tT$ is   topological. Since every object is fibrant, this means every $({\cX},{\cY})$-cofibration is an $h$-cofibration (\cite[Appendix A-1, Corollary 1.26]{schwede}) and by using the gluing lemma (\cite[Appendix A-2, Proposition 2.6]{schwede}) one gets that this model structure on $M\tT$ is left proper as well.
\subsection{Some remarks on generalizations}\label{ssect:generalizations}
Here we discuss two possible generalizations of our results. In order to avoid repetition, we only give rough sketches of the statements and ideas for proofs. 
\subsubsection*{$M$-objects in a model category}   An important  generalization of our model structure is to replace the ground category by a more general model category $\frC$, so that for a group $G$, the fixed point model structures on $G\frC:=[G,\frC]$ exists. A sufficient condition for existence of the fixed point model structures on $G\frC$ is given in \cite[Prop. 2.6]{stephan}. For $M$ and $(\cX,\cY)$  as in Section  \ref{sect:model-structures-on-m-spaces}, suppose that fixed point model structures exist on $G(N)\frC$ for every $N\in\cX$. Since $\frC$ is bicomplete, so are $M\frC$ and $G(M)\frC$. Thus, the right adjoint $\rnv{N}$ to the restriction $q^*_N:G(N)\frC\to N\frC$ exists. To generalize Theorem \ref{thm:Rmodel} on $M\frC$, it is enough to assume right adjoints to $q^*_N:G(N)\frC\to N\frC$ for all $N\in \cX$, $\Res^M_N$ and $\Delta_\cX$ preserve underlying fibrant replacements and path space objects. If the model structure on $\frC$ is not accessible but just cofibrantly generated, we also need to show that $q^*_N$ preserves smallness relative to cofibrations. The Quillen equivalence between $(\cX,\cY)$-model structure and induced orbit diagrams with projective model structure exists if (1) we have lemmas analogous to Lemma \ref{lem:rinv-preserve-pushouts} and Lemma \ref{lem:rinv-is-relatively-small} for $M\frC$ and (2) whenever $X$ is the domain or codomain of a generating (acyclic) cofibration in $\frC$, we have the existence of a natural isomorphism $(\rinv{K}{\Res^M_K(M\times_N q^*_N(G(N)/H))})^Q\otimes X \cong (\rinv{K}{\Res^M_K(M\times_N q^*_N(G(N)/H))}\otimes X)^Q$ for every $(N,H)$ and $(K,Q)$ in $\cO_{(\cX,\cY)}$. Under these conditions, the proof of the generalizations proceed similar to the proof of Theorem \ref{thm:Rmodel}.

\subsubsection*{Enriched case} There is a recently developed enriched version of the equivariant homotopy theory \cite{guillou2010enriched}. Therein, for a given bicomplete closed symmetric monoidal model category $\cV$, a certain type of group objects in $\cV$, called Hopf groups is introduced (as group objects in the symmetric monoidal category of cocommutative comonoids in $\cV$). For such a group object $G$, and a suitable $\cV$-enriched model category $\frC$, the category $G\frC$ of $G$-objects in $\frC$ (i.e., $\cV$-functors $G\to \frC$) admits a model structure generalizing the fixed point model structures, under certain conditions \cite[Sec. 3.2, Theorem 3.7]{guillou2010enriched}. Another possible generalization of the content of the present paper is to create model structures on $M\frC$, for sufficiently nice monoid objects $M$, where the homotopy theory is induced from the model structures of   \cite{guillou2010enriched}. However, group completions in the enriched setting often do not exist (e.g., for $\cV=\tT$ the algebraic group completions does not need to have a canonical topology making $q_M$ continuous). Even if there is an enriched group completion of a monoid object $M$ in $\cV$, we need to know that it is a Hopf group in $\cV$ and we need to find the set submonoids of $M$ (i.e., regular monomorphisms into $M$) whose group completions are Hopf groups. Several other conditions on $\cV$ and $\frC$ need to be imposed for existence of nice enriched Kan extensions along the group completions and suitability of these Kan extension for transfer arguments. It is difficult to determine all these conditions at once and most of the required work do not belong to general framework of this paper. On the other hand, this is a natural question that arise from the content of present paper. We leave these aspects for future work. The existence of such model categories would allow us, for example, to study modules over an associative algebra over a commutative ring by the enriched equivariant homotopy theories, as defined in \cite{guillou2010enriched}, of modules over a family of Hopf algebras.

\subsection{Some applications of $(\cX,\cY)$-model structure}\label{ssect:consequences}

Much of the constructions in standard equivariant homotopy theory  for group actions can readily be carried out for monoid actions in the setting of the homotopy theory of Theorem \ref{thm:Rmodel}.
Here we first give some immediate implications of the $(\cX,\cY)$-model structure being cofibrantly generated and then give some consequences of Theorem \ref{thm:Rmodel} analogous to the applications given for group actions in \cite{elmendorf}.

\subsubsection{$M_{(\cX,\cY)}$-$CW$-complexes and Whitehead Theorem}\label{sssect:cwcomplexes}
Let $(\cX,\cY)$ be as in Theorem \ref{thm:Rmodel}.  The $CW$-complexes can be constructed in the obvious way. A $M_{(\cX,\cY)}$-$CW$-complex $X$ is a sequential colimit of a $\omega$-sequence $c_\alpha:X_{\alpha}\to X_{\alpha+1}$ such that $X_{0}=\coprod_i M\times_{N_i} q^*_{N_i}(G(N_i)/H_i)$ for some $N_i\in \cX,\ H_i\in \cY_N$ for some indexing set and  $c_\alpha$ fits into a pushout diagram of the form
$$\xymatrix{
	\displaystyle 	\coprod_i M\times_{N_i} q^*_{N_i}(G(N_i)/H_i) \times \bbS^{n-1}  \ar[d]_{\coprod_i Id\times i_n} \ar[r]^{} &		X_{\alpha}\ar[d]^{c_\alpha}\\
	\displaystyle  \coprod_i M\times_{N_i} q^*_{N_i}(G(N_i)/H_i) \times \bbD^{n}  \ar[r]^{} & 	X_{\alpha+1}}.$$
for pairs $(N_i,H_i)$ varying over object of $\cO_{(\cX,\cY)}$.

An $M$-homotopy is a homotopy with respect to the interval $[0,1]$ with the trivial $M$-action, i.e., given $f_0,f_1:X\to X'$ an $M$-homotopy between $f_0$ and $f_1$ is an $M$-map $H:X\times [0,1]\to X'$ such that $H(-,0)=f_0$ and $H(-,1)=f_1$. It is straightforward to check that an $M$-homotopy equivalence is a weak equivalence in the $({\cX},{\cY})$-model structure for every $(\cX,\cY)$. The following corollary directly follows from  Theorem \ref{thm:Rmodel} and general properties of cofibrantly generated model categories (\cite[Lem. 4.24]{dwyer}).
\begin{corollary}\label{cor:whitehead}
	A weak equivalence $f:X\to X'$ in the  $({\cX},{\cY})$-model structure between $M_{(\cX,\cY)}$-$CW$-complexes is an $M$-homotopy equivalence. 
\end{corollary}
This is a form of   Whitehead Theorem for the model structure of Theorem \ref{thm:Rmodel}. In particular, $M$-homotopy type of an $M$-space  $X$ that is $M$-homotopy equivalent to an $M_{(\cX,\cY)}$-$CW$-complex reduces to homotopy type of the presheaf  $\frX(X):(N,H)\mapsto (\rinv{N}{(\Res^M_N(X))})^H$ for   $N\in \cX$ and $H\in \cY_N$ in the projective model structure of presheaves.

\subsubsection{A remark on Elmendorf's construction}\label{sssect:barconstruction}
Define a functor $$Y:{\cO_{(\cX,\cY)}} \to M\tT : (N,H)\mapsto M\times_N q^*_N(G(N)/H).$$ For $F: {\cO_{(\cX,\cY)}}^{\opp}\to \tT $, define $\Psi F$ to be the $M$-space $\bB({Y,{\cO_{(\cX,\cY)}},F})$, by the two-sided categorical bar construction, see \cite[Ch. V, 2.1]{may}.   The $M$-space $\bB({Y,{\cO_{(\cX,\cY)}},F})$ is the geometric realization of the simplicial $M$-space $\bB_*({Y,{\cO_{(\cX,\cY)}},F})$ whose $k$-simplices are given by  \begin{align*} \bB_k({Y,{\cO_{(\cX,\cY)}},F}) &=\coprod_{s:[k]\to \cO_{(\cX,\cY)}} M\times_{N_0} q^*_{N_0}(G(N_0)/H_0)\times F({N_k,H_k}) 
\end{align*} 
where $s(i)=(N_i,H_i)$. Since $\rnv{M}$ preserves coproducts, we have $$\bB_*({{\rnv{M}}{Y},{\cO_{(\cX,\cY)}},F})\cong\rinv{M}{\bB_*(Y,{\cO_{(\cX,\cY)}},F)}$$  as simplicial $G(M)$-spaces, where an element $(\sigma,t)\in\bB_k({{\rnv{M}}{Y},{\cO_{(\cX,\cY)}},F})$	is mapped to $s_\sigma\in \rinv{M}{\bB_k(Y,{\cO_{(\cX,\cY)}},F)}$ with $s_\sigma(g)=(\sigma(g),t)\in \bB_k(Y,{\cO_{(\cX,\cY)}},F).$

Assume that   $F({N,H})$ is a $CW$-complex for each $(N,H)$ in $\cO_{(\cX,\cY)}$.  The $M$-action on $\bB_k({Y,{\cO_{(\cX,\cY)}},F})$ is defined only on the first coordinate, implying that the $M$-action on $\bB_*({Y,{\cO_{(\cX,\cY)}},F})$ is simplicial; and thus, passes through the geometric realization  (see also \cite[proof of Thm.1]{elmendorf}).  Then the $M$-space  $M\times_{N_0} q^*_{N_0}(G(N_0)/H_0)\times F({N_k,H_k})$ is a $M_{(\cX,\cY)}$-$CW$-complex. Due to  \cite[Lem. 5.2.1 and Rem. 5.2.2]{riehl2014categorical} (see also \cite[Ex. 23.8]{shulman}) the simplicial $M$-space $\bB_*(Y,{\cO_{(\cX,\cY)}},F)$ is Reedy cofibrant, so that  $\bB({Y,{\cO_{(\cX,\cY)}},F})$ is a cofibrant $M$-space (here note also that the model structure on $M\tT$ is topological). This also implies that, when filtered by simplicial degree, the inclusions of skeleta are ${(\cX,\cY)}$-cofibrations, since the pushout products of latching maps $\ell_n:L_n(\bB_*({Y,{\cO_{(\cX,\cY)}},F})) \to \bB_n({Y,{\cO_{(\cX,\cY)}},F})$ and $i_n:{\mid}\partial\Delta^n{\mid}\to {\mid}\Delta^n{\mid}$ are ${(\cX,\cY)}$-cofibrations. By Lemmas \ref{lem:rinv-preserve-pushouts} and \ref{lem:rinv-is-relatively-small} we obtain  $$\bB({{\rnv{M}}{Y},{\cO_{(\cX,\cY)}},F})\cong\rinv{M}{\bB(Y,{\cO_{(\cX,\cY)}},F)}.$$
Since restrictions and fixed point functors commute with with the bar-construction (see, e.g., \cite[Proof of Thm. 1]{elmendorf}) we have an isomorphism 
$$\bB(\rinv{N}{\Res^M_N(Y-)^H},{\cO_{(\cX,\cY)}},F)\cong\rinv{N}{\Res^M_N(\bB({Y,{\cO_{(\cX,\cY)}},F}))}^H.$$
We have a natural bijection 
\begin{align*} 
	(\rinv{N}{\Res^M_N(Y(K,Q))})^H &\cong\rinv{N}{\Res^M_N(M\times_K q^*_{K}(G(K)/Q))})^H \\ 
	& \cong \Hom_M(M\times_N q^*_N(G(N)/H),M\times_K q^*_{K}(G(K)/Q))\\
	&\cong{\cO_{(\cX,\cY)}}((N,H),(K,Q)) ,
\end{align*} 
for every $(K,Q)$ in ${\cO_{(\cX,\cY)}}$. Combining with the isomorphism established above, we have  $$
(\rinv{N}{\Res^M_N\bB({Y,{\cO_{(\cX,\cY)}},F})})^H\cong\bB(\cO_{(\cX,\cY)}((N,H),-),{\cO_{(\cX,\cY)}},F)
$$
By \cite[Ch. V. 2.2]{may}, we have a homotopy equivalence $$\varepsilon: \bB(\cO_{(\cX,\cY)}((N,H),-),{\cO_{(\cX,\cY)}},F) \to F(N,H),$$ which is natural in $(N,H)$. Being a left Quillen functor, $\Upsilon$ preserve weak equivalences between cofibrant objects. Thus, $\Upsilon \varepsilon : \Psi F \to \Upsilon F$ is an weak equivalence for any cofibrant $F$.

Now, assume that $X$ is an $M_{(\cX,\cY)}$-$CW$-complex for which  $\rinv{K}{\Res^M_K(X)})$ is a $G(K)$-$CW$-complex with respect to $\cY_K$-model structure. Then $\rinv{K}{\Res^M_K(X)})^H$ is of the homotopy type of a $CW$-complex.   Letting $F=\frX  X$ in the above natural weak equivalence, we obtain a weak equivalence  $\Upsilon \varepsilon: \Psi \frX
X \to  X$. Therefore, following the steps in proof of Theorem 2 in \cite{elmendorf}, we get a natural bijection 
\begin{equation}\label{eqn:elmendorf}[X,\Psi F]_{(\cX,\cY)}\cong [\frX X, F]_{\cO_{(\cX,\cY)}},
\end{equation}
where $[-,-]_{(\cX,\cY)}$ and $[-, -]_{\cO_{(\cX,\cY)}}$ homotopy classes of maps in the respective homotopy categories.
\begin{remark}
	A sufficient condition for $\rinv{K}{\Res^M_K(X)})$ being a $G(K)$-$CW$-complex (with respect to $\cY_K$-model structure) for an $M_{(\cX,\cY)}$-$CW$-complex $X$ is  that   the subspace topology of the prodiscrete topology on hom-sets of ${\cO_{(\cX,\cY)}}$ agrees with the discrete topology, i.e., $(\rinv{K}{M\times_Nq^*_N(G(N)/H)})^Q$ is discrete  for every $(N,H),\ (Q,K) \in {\cO_{(\cX,\cY)}}$. For every monoid $M$, there are pairs of collections satisfying this condition and for commutative cancellative monoids, this condition holds for every pair of collections (see the last paragraph of \ref{ssect:induced-orbit-category}). However, an infinitely generated and infinitely presented monoid can be constructed as a counterexample, since $A$ is discrete does not in general implies $\rinv{M}{A}$ is so.
\end{remark}
\subsubsection{Classifying spaces for families of pairs}\label{sssect:classifying-spaces}
By using Theorem \ref{thm:Rmodel}, one can define classifying space for a given pair of collections just as in \cite{elmendorf}. Let $(\cX,\cY)$ be a pair of collections. 
A pair of subcollections $(\cX_0,\cY_0)$ in $(\cX,\cY)$ is called a \emph{pair of families relative to $(\cX,\cY)$} if  the full subcategory $\cO_{(\cX_0,\cY_0)}$ is a sieve in $\cO_{(\cX,\cY)}$ in the sense \cite[Def. 6.2.2.1]{lurie}. Then the functor $E_{(\cX,\cY)}:{\cO_{(\cX,\cY)}}^{\opp}\to \tT$  given by  
$$E_{(\cX_0,\cY_0)}(N,H)=  \begin{cases}
	*  & \text{if } (N,H)\in \cO_{(\cX_0,\cY_0)} \\
	\emptyset & \text{if }  (N,H)\notin \cO_{(\cX_0,\cY_0)}  
\end{cases}$$
is well defined.  We say an $M_{(\cX,\cY)}$-$CW$-complex $X$ is \emph{$(\cX_0,\cY_0)$-isotropic relative to $(\cX,\cY)$} if for every $(N,H)\in \cO_{(\cX,\cY)}$ $(\rinv{N}{\Res^M_N(X)})^H\neq \emptyset$ implies $(N,H)\in \cO_{(\cX_0,\cY_0)}$.

Following  \cite[Sec. 2]{elmendorf}, one defines \emph{the classifying space for $(\cX_0,\cY_0)$} by the two-sided categorical bar construction $\bB({Y,{\cO_{(\cX,\cY)}},E_{(\cX_0,\cY_0)}})$, where $Y$ is the functor given by $Y(N,H)= M\times_N q^*_N(G(N)/H)$.

From Theorem \ref{thm:Rmodel} (and Section \ref{sssect:barconstruction}) for every $(\cX_0,\cY_0)$-isotropic $M_{(\cX,\cY)}$-$CW$-complex $X$ satisfying Equation \ref{eqn:elmendorf} in \ref{sssect:barconstruction}, $[X,\bB({Y,{\cO_{(\cX,\cY)}},E_{(\cX_0,\cY_0)}})]_{(\cX,\cY)}$ is a singleton, which is naturally isomorphic to $[\frX X,E_{(\cX_0,\cY_0)}]_{\cO_{(\cX,\cY)}}$. In fact, if $(N,H)\notin\cO_{(\cX_0,\cY_0)}$ then $\frX X(N,H)=(\rinv{N}{\Res^M_N(X))}^{H}= \emptyset$.

\subsubsection{Eilenberg-Maclane $M$-spaces}\label{sssect:eilenberg-maclane-spaces}
Similarly, one defines Eilenberg-Maclane $M$-spaces for in $({\cX}^{all},{\cY}^{all})$-model category  by just following the footsteps of \cite[Sec. 2]{elmendorf}. Here $({\cX}^{all},{\cY}^{all})$ is the pair of collections with $\cO_{({\cX}^{all},{\cY}^{all})}=\cO(M:q)$. Given an $\cO(M:q)$-group $\Lambda$ (i.e., a functor $\Lambda:\cO(M:q)\to \Grp$) and a positive integer $n$ (where $\Lambda$  is abelian for $n>1$), an Eilenberg-Maclane $M$-space for $\Lambda$ is an  $M_{({\cX}^{all},{\cY}^{all})}$-$CW$-complex $X$ such that for each object $(N,H)$ in $\cO(M:q)$ we have $(\rinv{N}{\Res^M_N X})^H\cong K(\Lambda(N,H),n)$ and $\Lambda = \pi_n\circ \frX X$. The construction can be obtained by the Bar construction for the compositions $K(-,n)\circ \Lambda$ just as in  \cite[Sec. 2]{elmendorf} and is unique up to $M$-homotopy due to Corollary \ref{cor:whitehead} above.

\subsubsection{Bredon-like cohomology of $M$-spaces}\label{sssect:mcohomology}
We use an approach similar to definition of Bredon cohomology in \cite{moerdijk1993equivariant} to define cohomology of $M$-spaces for the homotopy theory of Theorem \ref{thm:Rmodel}.  Let $\mathbf{S}_*:\tT\to \sS$ be the singular functor. Given a space $A$ there exist a category $\Simp(A)$, called the simplex category, whose objects are maps $\sigma:\Delta^n\to \mathbf{S}_*A$ from a standard simplicial simplex to $\mathbf{S}_*A$ and whose morphisms are obvious commuting triangles (i.e., the comma category for inclusions of $\mathbf{S}_*A$ and $\Simp$ into $\sS$). This construction defines a functor $\Simp:\tT\to \Cat$. Note that $\mathbf{N}\Simp(A)$ is weak equivalent to $ \mathbf{S}_*A$ as simplicial sets, where $\mathbf N$ denotes the nerve functor (see, e.g., \cite[Thm. 19.9.3]{hirschhorn}).

For a given small category $\cC$ and a functor $F:\cC^{\opp}\to \Cat$ the Gr\"othendieck construction over $F$ is a small category $\int F$ whose objects are pairs $(c,d)$ where $c$ is an object in $\cC$ and $d$ is an object in  $Fc$ and a morphism between such pairs is a pair morphisms $(f,g):(c,d)\to (c',d')$ where $f$ is a morphism in $\cC$ and $g:d\to F(f)(d')$ a morphism in $Fc$  (see, e.g., \cite{thomason} or \cite{elephant}). The Gr\"othendieck construction is equipped with a canonical projection given by $p_\cC:\int F\to \cC:(c,d)\mapsto c$.

For a given $M$-space $X$ define the $M$-equivariant simplex category $\Simp_M(X)$  as the Gr\"othendieck construction for the composition $\Simp\frX(X):{\cO_{(\cX,\cY)}}^{\opp}\to \Cat$, i.e., $$\Simp_M(X)=\int \Simp \frX(X)=\int_{(N,H)} \Simp ((\rinv{N}{\Res^M_N X})^H).$$
Similarly, letting $\pi:\tT\to \Cat$ denote the fundamental groupoid functor, one defines, $\pi_M(X)$, the $M$-equivariant fundamental groupoid (for the homotopy theory in Theorem \ref{thm:Rmodel}) as the Gr\"othendieck construction for the composition $\pi\frX(X):{\cO_{(\cX,\cY)}}^{\opp}\to \Cat$, i.e., $$\pi_M(X)=\int \pi \frX(X).$$ Then we have functors $$\Simp_M(X)\stackrel{\omega_X}\to \pi_M(X)\stackrel{p_X}\to \cO_{(\cX,\cY)},$$ where the first functor is induced by the natural transformation $\Simp(-)\Rightarrow \pi(-)$ as described in \cite[pp. 267]{moerdijk1993equivariant} on fibers. We define cohomology theories in this setting by following \cite[Sec. 1 and 2]{moerdijk1993equivariant}: for a given local coefficient system $L$, i.e., a functor $L:\pi_M(X)\to \Ab$, we define the $M$-cohomology of $X$ with coefficients in $L$, $H^n_M(X,L)$, as cohomology of the category $\Simp_M(X)$ with respect to the coefficients $L$, i.e., cohomology of the cochain complex $C_{L\omega_X}^*$ with $$C_{L\omega_X}^n=\prod_{\varsigma\in \mathbf{N}\Simp_M(X)_n}L\omega_X(\varsigma(0))$$ where $\mathbf{N}\Simp_M(X)$ denote the nerve of $\Simp_M(X)$ and $\varsigma=\varsigma(0)\to \varsigma(1)\to \cdots \varsigma(n)$ is an $n$-simplex in $\mathbf{N}\Simp_M(X)$. The differentials of this cochain complex are obtained by taking the alternating sum of face maps in the nerve, as usual.

If a coefficient system $L$ is constant, i.e., factors through the projection $p_X:\pi_M(X)\to \cO_{(\cX,\cY)},$ then the cohomology theory as defined above is homotopy invariant. In fact, if $f:X\to Y$ is a weak equivalence between  $M_{(\cX,\cY)}$-$CW$-complexes, then $\mathbf{N}\Simp((\rinv{N}{\Res^M_N f})^H)$ is a weak equivalence of simplicial sets, which implies $\mathbf{N}\Simp_M(f)$ is a weak equivalence of simplicial sets. Thus, $H^n_M(f,L)$ is an isomorphism for every $n\in \bbN$.
\subsubsection*{Serre spectral sequence and homotopy invariance of $H^n_M$}
For the homotopy invariance of $H^n_M$  with local coefficients, we use adaptations of the arguments in \cite[Sec. 2.3 and 2.4]{moerdijk1993equivariant}. We construct the Serre spectral sequence for $H^n_M$ and $(\cX,\cY)$-fibrations. Let $X$ be an $M_{(\cX,\cY)}$-$CW$-complex. By definition of $\Simp_M(X)$ combined with \cite{thomason}, the weak equivalence $$\displaystyle\hocolim_{\Simp((\rinv{N}{\Res^M_N(X)})^H)}pt\cong (\rinv{N}{\Res^M_N(X)})^H$$ (see, e.g., \cite{goerss}) give rise to a weak equivalence $$\displaystyle\hocolim_{\Simp_M(X)}Yp_{\cO_{(\cX,\cY)}}\cong X.$$
Here $Y$ is the functor defined as in \ref{sssect:barconstruction}. Given a  $(\cX,\cY)$-fibration $f:E\to X$ in $M\tT$ and an object $((N,H),\sigma:\Delta^n\to \mathbf{S}_*(\rinv{N}{\Res^M_N X})^H)$ in $\Simp_M(X)$, one obtains a homotopy pullback diagram 
$$\xymatrix{
	f^*(\bar\sigma)  \ar[d]_{} \ar[r]^-{ \tilde\sigma} &		E\ar[d]^{f}\\
	M\times_N q^*_N(G(N)/H) \times \Delta^{n}  \ar[r]_-{\bar\sigma} & 	X}$$
Here $\bar\sigma$ is obtained from the isomorphism $$\Hom_M(M\times_N q^*_N(G(N)/H)\times \Delta^n,X)\cong \Hom(\Delta^n,(\rinv{N}{\Res^M_N X})^H).$$ Define a functor $T:\Simp_M(X)\to \tT$ by $T(((N,H),\sigma))=f^*(\bar\sigma)$. Then, as above, the weak equivalence on fibers (see \cite[pp. 247]{goerss}) give rise to a weak homotopy equivalence $$\displaystyle\hocolim_{\Simp_M(X)}T\cong E.$$ The Serre spectral sequence for $H_M^*$ follows as a direct consequence of the homotopy colimit spectral sequence \cite[Ch. IV, 5.1]{goerss} (or directly as in \cite[Thm. 3.2]{moerdijk1993equivariant}). For a given local coefficient system $L$ for $E$ this natural spectral sequence reads $$E_2^{p,q}=H^p_M(X,H^q_M(f^{*},L))\Rightarrow H^{p+q}_M(E,L)$$ where 
$H^q(f^{*},M):\Simp_M(X)\to \Ab$ is the coefficient system given by  $$(N,H,\sigma)\mapsto H^q_M(f^*(\bar\sigma),L\omega_X\Simp_M(\tilde\sigma)).$$

The homotopy colimit spectral sequence follows from the skeletal filtration of the bar construction whose realization is the homotopy colimit. Let $E=X$ and $f=id$. We construct a bisimplicial set $X(*,*)$ by the two sided simplicial bar construction $\bB_*(\mathbf{S}_*Yp_{\cO_{(\cX,\cY)}},\Simp_M(X),*)$ so that $\diag(X(*,*))\cong \mathbf{N} X\cong \mathbf{N}\Simp_M(X)$, e.g. \cite[Cor. 5.13]{riehl2014categorical}.  We have
\begin{align*} X(p,q)&=\coprod_{s:[q]\to \Simp_M(X)}\mathbf{S}_*(M\times_{N_0}q^*_{N_0}(G(N_0)/H_0))_p \end{align*}  
An element of $X(p,q)$ can be realized as a triple $$(s,u,\sigma) \in (\mathbf{N}\cO_{(\cX,\cY)})_q\times(\mathbf{N}\Delta)_p  \times \Hom( \Delta^{u(p)},(\rinv{N}{\Res^M_N X})^H).$$
Let $C^{p,q}(X,L)$ be the associated double complex so that $$H^n(\Tot(C^{*,*}(X,L)))\cong H^n_M(X,L),$$ (compare \cite[Proof of 2.2, Eqn. (23) and (25)]{moerdijk1993equivariant}). The rows of $C^{p,q}(X,L)$ are exact, i.e., the associated (horizontal) spectral sequence collapses. The  vertical spectral sequence is of the form 
$$E_1^{p,q}=\prod_{s\in(\mathbf{N}\cO_{(\cX,\cY)})_q }H^p(\frX X(s(q)), L^s)\Rightarrow H^{p+q}_M(X,L),$$ where $L^s$ is the functor given by the composition of $L$ with the endofunctor on $\Simp_M(X)$ induced by the composition of maps in $s$. The cohomology in the product denotes the ordinary twisted cohomology. Since the weak equivalences in $M\tT$ are defined via $\frX$ and the ordinary twisted cohomology is homotopy invariant, homotopy invariance of $H^{n}_M(-,L)$ follows.

\subsection{Concluding remarks on further applications}\label{sssect:remarks-on-applications} 
The difficulty of studying monoid actions over group actions is on par with the difficulty of studying all maps over studying only isomorphisms. For example a pair of distinct orbits of a monoid action do not need to be disjoint, so there is no efficient orbit decomposition for a monoid action. The functor $\rnv{M}$ sorts out invariant subspaces of an $M$-space on which $M$ acts by homeomorphisms and maximal with respect to this property. Since it is a right adjoint to the forgetful functor $q_M^*$, we can say that $\rnv{M}$ transforms a logically irriversible process into a reversible one by preserving the essential information as much as possible. Here note that in nature almost all processes are irreversible.

Recall that for an $M$-space $X$, the image of $\epsilon_X$ is the maximal $M$-invariant subspace of $X$ (see Proposition \ref{prop:imepsilonisMinvariant}). By restriction to submonoids $N\leq M$ and applying $\rnv{N}$, we can in fact decompose $X$ further into $N$-invariant subspaces (with respect to the restricted action of $N$) on which $N$ acts by homeomorphisms. Therefore, the material discussed in Sections \ref{sect:preliminaries} and \ref{sect:model-structures-on-m-spaces} has potential applications on dynamical properties of irreversible  systems that  depend on invariance, orbits types and transitivity of the actions (such as,  recurrence, stability,  observability, controllability, see also \cite{souza2020chaos,mittenhuber1995semigroup,sain1969invertibility}). We briefly sketch some  immediate ones below.

\subsubsection*{Attractors}\label{sssect:attractors}
Roughly speaking, an attractor of a dynamical system is a subset that is invariant under the action that contains limits points of a neighborhood. For the classical notion we refer to \cite{milnor1985concept}. Most basic case is when we consider dynamics of single map $f:X\to X$. For $B\subseteq X$ compact subset such that $f(B)\subseteq \overset{\circ}B$, a (trapped) attractor for this system is the intersection $A=\bigcap_{n\in \bbN}f^{\circ n}(B)$ (here $f^{\circ n}$ is the $n$-fold composition of $f$). It can be seen that $f(A)=A$, so this set is invariant under the $\bbN$-action generated by $f$, i.e. the action given by $n\cdot x=f^{\circ n}(x)$. Besides, it is the maximal one in $B$. Therefore $A=\epsilon_B(\rinv{\bbN}{B})$. In fact, $x\in A$ if and only if $\forall n\in \bbN$, $\exists b\in B$, $f^{\circ n}(b)=n\cdot b=x$, and this holds if and only if $\exists \sigma\in \epsilon_B(\rinv{\bbN}{B}) $ with $x=\sigma(1)$ (here note that $b=\sigma(-n)$). 

Later, the notion is generalized to semigroup actions, e.g. see  \cite{souza2017existence}. There, the notion of global (uniform) attractor is also defined. For an $M$-space $X$  its global uniform attractor, if exists, can be obtained as the reunion of its $M$-invariant compact subsets \cite[Thm. 2.4]{souza2017existence}. Thus, the global uniform attractor can be seen as the image of a compact $G(M)$-invariant subspace of  $\rinv{M}{X}$ under $\epsilon_X$. In fact, if $A\neq \emptyset$ is the global uniform attractor, then by definition it is a  compact $M$-invariant subspace of $X$. By Lemma \ref{lem:rinv-is-closed}, $\rinv{M}{A}$  can be regarded as a compact subspace of $\rinv{M}{X}$. By  Proposition \ref{prop:imepsilonisMinvariant} we have $\epsilon_X(\rinv{M}{A})=A$.  We also have $\epsilon_A:q^*_M(\rinv{M}{A})\to A$ is a weak equivalence in the model structure of Proposition \ref{prop:preliminarymodels} part I. This helps us to determine homotopical properties of such attractors.

We also make following observation for one parameter systems. For the following remarks `a system'  means a $\bbN$-action on a space.
\subsubsection*{Periodicity and immortality}\label{sssect:perimm} 
Periodicity and immortality are two important dynamical properties of a given system. These phenomena are studied in \cite{kari2008periodicity} for reversible Turing machines and \cite{hooper1966undecidability} for irreversible ones.  For a given irreversible system generated by a (non-invertible) map $\phi:X\to X$, $\rinv{\bbN}{X}$ is a reversible system consisting of periodic and immortal components. The type of these components can be understood easily via the equivariant homotopy type. In fact, immortal components have empty fixed point sets for any nontrivial subgroup $H\leq \bbZ$. For example if $(\bbN,X)$ is as in Example \ref{ex:RUS}, then  $\rinv{\bbN}{X}$ consists of a (uniformly) periodic and an immortal component, if $(\bbN,Z)$ is as Example \ref{ex:RuR} then $(\bbZ,\rinv{\bbN}{Z}\cong \bbR\amalg \bbR\amalg \bbR)$ contains three immortal components but not a periodic one.

\subsubsection*{Time-reversal symmetries}\label{sssect:time-reversal} 
Time-reversal symmetries of dynamical systems are important types of basic symmetries that appear in natural sciences. For a general survey on time-reversal symmetries we refer to \cite{lamb1998time}. Let $\phi:X\to X$ be a given homeomorphism generating a discrete dynamical system $(\bbZ,X)$. A time-reversal symmetry of  $(\bbZ,X)$ is a homeomorphism  $r:X\to X$ so that $r\circ \phi = \phi \circ r$. Several examples and consequences of the notion is discussed in  \cite{lamb1998time}. In order to discuss time-reversal symmetries of a system, we need the system to be reversible. Thus, for a general irreversible system $(\bbN, X)$ given by an action of $\bbN$ we cannot discuss time-reversal symmetries unless $\bbN$ acts by homeomorphisms. On the other hand, we can always discuss time-reversal symmetries after applying $\rinv{\bbN}{-}$. For example, if $(\bbN,Z)$ is as Example \ref{ex:RuR}, then the derived system $(\bbZ,\rinv{\bbN}{Z}\cong \bbR\amalg \bbR\amalg \bbR)$ admits a time-reversal symmetry given by $r:t\mapsto -t$ where $t$ belong to either of the three connected components.

\bibliographystyle{plain}
\bibliography{els-monactbygrpact-refs}

\end{document}